%% file: Simplicial_V2.tex
\documentclass[twoside, 11pt]{amsart}

\input{package.tex}

\input{authors.tex}

\begin{document}

\maketitle

\begin{abstract}
In this paper, we settle the homotopy properties of the $\infty$-morphisms of homotopy (bial)\-ge\-bras over properads, i.e. algebraic structures made up of operations with several inputs and outputs. We start by providing the literature with characterizations for the various types of $\infty$-morphisms, the most seminal one being the equivalence between $\infty$-quasi-isomorphisms and zig-zags of  quasi-isomorphisms which plays a key role in the study of the formality property. 
We establish a simplicial enrichment for the categories of gebras over some cofibrant properads together with their $\infty$-morphisms, whose homotopy category provides us with the localisation with respect to $\infty$-quasi-isomorphisms. 
These results extend to the properadic level known properties for operads, but the lack of the rectification procedure in this setting forces us to use different methods. 
\end{abstract}

\setcounter{tocdepth}{2}
\tableofcontents

\section*{Introduction}

\paragraph{\bf Differential graded bialgebras} 
In the present article, we pursue the study of the homotopical properties of types of differential graded bialgebras initiated in \cite{HLV19}. 
Let us recall that the homotopy theory of types of differential graded algebras is now well understood, thanks to the development of the suitable operadic calculus prompted by the Koszul duality, see \cite{LodayVallette12} for instance. An operad is a fundamental object which encodes operations with several inputs but one output. One can encode categories of algebraic structures made up of operations with several inputs and several outputs by the notion of a \emph{properad}, which also admits a Koszul duality theory \cite{Vallette07}. Since properads can govern types of algebras, coalgebras, and bialgebras, we simply call \emph{gebras} the categories of algebraic structures encoded by properads after J.-P. Serre \cite{Serre93}. The deformation theory of gebras was developed in \cite{MerkulovVallette09I, MerkulovVallette09II} and the notion of an $\infty$-morphism of gebras was introduced recently in \cite{HLV19}. Their obstruction theory, the  ubiquitous homotopy transfer theorem with effective formulas, and  the quasi-invertibility of $\infty$-quasi-isomorphisms were all settled on the properadic level in \emph{loc. cit.}.\\

\paragraph{\bf Formality property} 
When two differential graded gebras are related by a quasi-isomorphism, i.e. a morphism inducing an isomorphism on homology, 
they share common homotopy properties. But given a quasi-isomorphism of differential graded gebras, there does not necessarily exist a quasi-isomorphism in the other direction which preserves the respective algebraic structures. In order to get a well-defined equivalence relation, one is obliged to consider zig-zag of quasi-isomorphisms. The most seminal example of such a situation is the \emph{formality property} which amounts to requiring that a differential graded gebra is related to its induced homology gebra by a zig-zag of quasi-isomorphisms. In this particular case, the homotopy properties of the original differential graded gebra are direct consequences of its more simple homology gebra. 
The two most seminal examples of formality properties of differential graded algebras are given by the commutative algebra of differential forms on compact K\"ahler manifolds due to Deligne--Griffiths--Morgan--Sullivan \cite{DGMS75} and by the
differential graded Lie algebra structure on the Hochschild cochain complex of the algebra of smooth functions on manifolds
due to M. Kontsevich \cite{Kontsevich03}. \\

\paragraph{\bf Zig-zags of quasi-isomorphisms vs $\infty$-quasi-isomorphisms} 
In the former case, a zig-zag made up of only two maps and one intermediate differential graded commutative algebra is directly constructed using the powerful $d d^c$-lemma. This is a very favorable situation and the general case is difficult to establish in practice since it amounts to constructing several intermediate differential graded algebras. Considering assignments is more simple in general and this is the strategy followed in the latter case. It relies on the fondamental result saying that the existence of a zig-zag of quasi-isomorphisms of differential graded algebras is equivalent to the existence of a \emph{direct} $\infty$-quasi-isomorphism, see \cite[Theorem~11.4.9]{LodayVallette12}. 
The first main result of the present paper is a generalisation of this property to differential graded gebras over properads of the form $\mathcal{P}_\infty=\Omega \calC$, that is given by the cobar construction of differential graded conilpotent coproperads.

\begin{theointro}[\bf \ref{thm:MainInftyQi}]
Two $\mathcal{P}_\infty$-gebra structures $(A,\alpha)$ and $(B, \beta)$ are $\infty$-quasi-isomorphic if and only if they are related by a zig-zag of quasi-isomorphisms of $\mathcal{P}_\infty$-gebras: 
\[
\begin{array}{c}
\exists\  \text{$\infty$-quasi-isomorphism} \\
\begin{tikzcd}[column sep=normal]
	(A,\alpha) \ar[r,squiggly,"\sim"] 
	& (B, \beta)
\end{tikzcd} 
\end{array}
\Longleftrightarrow
\begin{array}{c}
 \exists\  \text{zig-zag of quasi-isomorphisms} \\
\begin{tikzcd}[column sep=small]
	(A, \alpha) \ar[r,"\sim"] 
	& \point
	& \point \ar[l,"\sim"'] \ar[r,dotted, no head]
	&\point
	& \point \ar[l,"\sim"'] \ar[r,"\sim"] 
	& (B, \beta)
	\end{tikzcd}  \ .
\end{array}
\]
\end{theointro}

This statement is the key result which opens the door to formality properties for homotopy gebras. 
The new range of possible applications is rather wide since it now allows one to deal with 
associative bialgebras (using the bar construction for $\calC$), (involutive) Lie bialgebras \cite{CFL20}, (involutive) Frobenius bialgebras \cite{Yalin18}, 
Double Poisson bialgebras \cite{Leray19protoII, LV23}, Pre-Calabi--Yau (bi)algebras \cite{KTV23, KTV25}, (balanced) infinitesimal bialgebras \cite{Aguiar00, Q23}, Airy structures \cite{KontsevichSoibelman18, BV25}, 
etc. 
It has already been used recently by C. Emprin \cite{Emprin24} in order to settle faithful obstruction classes for the  formality property and by 
Emprin--Takeda \cite{EmprinTakeda25} to establish the intrinsic (co)formality property of spheres over the rational numbers with respect to a properadic structure close to pre-Calabi--Yau algebras. 
We also establish here properties and characterisations for the other types of $\infty$-morphisms: when their first component is injective, surjective, equal to the identity or generic. \\

\paragraph{\bf Deformation theory of $\infty$-morphisms} To establish further homotopical properties for properadic $\infty$-morphisms, 
we apply the up-to-date methods of deformation theory. First, we solve the question of finding a (complete) homotopy Lie algebra which encodes 
$\infty$-morphisms of $\calP_\infty$-gebras under the Maurer--Cartan equation: it is given by a convolution type complete homotopy Lie algebra, see \cref{prop:MCInfty}. Recall that curved $\mathrm{L}_\infty$-morphisms, are $\infty$-morphisms of  homotopy Lie algebras which are not necessarily pointed. This extended framework allows us to enrich the category of $\calP_\infty$-gebras with homotopy Lie algebras, see \cref{thm:EnrichCat}. \\

\paragraph{\bf Simplicial category of homotopy gebras}
It remains to apply the Deligne--Hinich integration functor $\MC_\bullet$ which assigns Kan complexes to (complete) homotopy Lie algebras to get a simplicial category of $\calP_\infty$-gebras whose $0$-simplices are the $\infty$-morphisms, see \cref{prop:SimplicialCat0Simpl}. 
We show that this simplicial category is well suited to the homotopy properties of $\infty$-morphisms. First, its homotopy category provides us with both localisations of the category of $\calP_\infty$-gebras with respect to quasi-isomorphisms and to $\infty$-quasi-isomorphisms.

\begin{theointro}[\bf \ref{thm:HoCatViaSimpl}]
The canonical functor 
\[\infty\text{-}\catofgebras{\calP_\infty} \to \pi_0\left(\Delta\textsf{-}\catofgebras{\calP_\infty}\right)\]
is the universal functor which sends quasi-isomorphisms (respectively $\infty$-quasi-isomorphisms) to isomorphisms: 
\[\infty\text{-}\catofgebras{\calP_\infty}\left[\mathsf{qi}^{-1}\right] \cong   
\infty\text{-}\catofgebras{\calP_\infty}\left[\infty\text{-}\mathsf{qi}^{-1}\right] \cong   
\pi_0\left(\Delta\textsf{-}\catofgebras{\calP_\infty}\right)~.\]
\end{theointro}

Then, this simplicial category structure allows us to characterise $\infty$-quasi-isomorphisms in terms of weak equivalences of Kan complexes. 

\begin{theointro}[\bf \ref{thm:CaraInftyQI}]
An  $\infty$-morphism $f$ of $\P_\infty$-gebras is an $\infty$-quasi-isomorphism if and only if 
the pullback map $\MC_\bullet\left(f^*\right)$ and/or the pushout map $\MC_\bullet\left(f_*\right)$ are weak equivalences of 
Kan complexes. 
\end{theointro}

Finally, the homotopy invariance property of the Deligne--Hinich integration functor \cite{DolgushevRogers15, RNV19}, which a far reaching generalisation of the celebrated Goldman--Millson theorem \cite{GoldmanMillson88}, gives a straighforward proof of a version of the homotopy transfer theorem for homotopy gebras (\cref{Thm:HTTnew}), which is very similar to the one obtained by B. Fresse \cite{Fresse10ter} with completely different methods. 
These two close versions of the homotopy transfer theorem are however rather different then to the one given in \cite{HLV19}, which is not  produced by abstract existence but by explicit formulas.\\

\paragraph{\bf Discourse on the method}
All the results given in this article are generalisations of operadic statements to the properadic level, but most of them are far from being automatic: proofs were expected since the introduction of the Koszul duality for properads \cite{Vallette07} in 2007 and since the simplicial enrichment of homotopy algebras \cite{DolgushevHoffnungRogers14} in 2015. In the operadic case, the equivalence between zig-zags of quasi-isomorphisms and $\infty$-quasi-isomorphisms was settled using the bar-cobar adjunction and the induced rectification procedure both for homotopy algebras and for $\infty$-morphisms \cite[Section~11.4.3]{LodayVallette12}. Such a method cannot hold for homotopy gebras: due to the lack of a free gebra functor, there exists no suitable bar and cobar constructions neither a similar rectification functor. In order to bypass this conceptual obstruction, we use arguments from a model category structure on $2$-colored differential graded properads (\cref{prop:ModelCatColoreddgProperads}) and from the theory of endomorphisms of diagrams of properads (\cite{Fresse10ter}). The overall strategy to develop the simplicial enrichment of the category of homotopy gebras is borrowed from Dolgushev--Hoffnung--Rogers \cite{DolgushevHoffnungRogers14} but its properadic extension cannot be established using the same proofs. On a technical side, adequate filtrations on trees are used in \emph{loc. cit.} in order to make the various spectral sequences to converge. In the present work, we had to introduce new types of filtrations on graphs, like the \emph{density filtration} (\cref{def:DensOp}) to conclude. 
Such more powerful filtrations are expected to be used fruitfully in future studies on properadic gebras. Then, the homotopy properties 
of the simplicial enrichment of homotopy algebras are proved in \cite{DolgushevHoffnungRogers14} using 
the fact that path objects of endomorphisms operads coincide with endomorphism operads with values in path objets. This key property is true since operads have only one output, but it cannot hold anymore for properads. 
 To overcome this last issue, we use refined model categorical arguments of endomorphisms of $2$-colored properads, see the proof of \cref{thm:HoCatViaSimpl}.  
Finally, it is worth noticing that the properadic calculus performed here is more conceptual than the operadic calculus used in \emph{loc. cit.} with the beneficial consequence for the present exposition to be more concise. \\
 
\paragraph{\bf Layout} In the first section, we start with recollections on the properadic calculus including the 
definition of $\infty$-morphisms for $\P_\infty$-gebras. Then, we settle a canonical cofibrant replacement of the 
$2$-colored properad encoding the data of two $\P_\infty$-gebras related by an $\infty$-morphism and we introduce a canonical very good cylinder for the dg properad $\P_\infty=\Omega \calC$. Both are then used to establish the various characterisations of 
$\infty$-isotopies, $\infty$-quasi-isomorphisms, $\infty$-monomorphisms, $\infty$-epimorphisms and $\infty$-morphisms. 
The second section begins by recalling the various notions of (complete/shifted) $\mathrm{L}_\infty$-algebras and their (continuous/curved) $\infty$-morphisms. The convolution type homotopy Lie algebra encoding pairs of $\P_\infty$-gebras related by an $\infty$-morphism is first described and then used to enrich the category of $\P_\infty$-gebras and $\infty$-morphisms. The third section opens with the definition of the Deligne--Hinich integration functor. It is then used to provide the category of $\P_\infty$-gebras and $\infty$-morphisms with a simplicially enriched structure. Its associated homotopy category is shown to give the 
category of $\P_\infty$-gebras localised at $\infty$-quasi-isomorphisms. We conclude with the characterisation of $\infty$-quasi-isomorphisms in terms of weak equivalences between Kan complexes of mapping spaces.
\\

\paragraph{\bf Conventions} 
We work over a ground field  $\k$ of characteristic $0$ and in the  underlying category of differential $\ZZ$-graded vector spaces, denoted by $\dgVect$. We use the homological degree convention, for which differentials have degree $-1$. We denote the symmetric groups by $\bbS_n$. 

\medskip

We use the same conventions on the properadic homotopical calculus as in \cite{HLV19}, 
the same conventions on the integration of complete shifted  homotopy Lie algebras as in \cite{RNV19}, 
and the same conventions on $\infty$-morphisms of complete shifted curved homotopy Lie algebras as in \cite{DSV18}. \\

\paragraph*{\bf  Acknowledgements} 
We would like to express our appreciation 
to David Chataur, Coline Emprin, Benoit Fresse, Joan Mill\`es, Victor Roca i Lucio, and Thomas Willwacher 
for enlightening discussions and V.A.V. for the help with the pictures. 
We are grateful to the referees for their careful reading and wise advises which helped improving the present text. 
The third author would like to thank the Institut des Hautes \'Etudes Scientifiques for the long term invitations and for the ideal working conditions there. 

\section{Homotopical properties of $\infty$-morphisms}

This section begins with recollections on properads and their categories of (bial)gebras from \cite{HLV19} including  the seminal notion of an $\infty$-morphism. In order to establish their homotopical properties, we encode them with a cofibrant 2-colored dg properad. The ultimate goal is to establish characterizations for the various types of $\infty$-morphisms ($\infty$-isotopies, $\infty$-quasi-isomorphisms, $\infty$-morphisms) in terms of zig-zag of strict morphisms. 

\subsection{Properadic calculus}\label{subsec:PropCal}
Here are a few recollections on the properadic homotopical calculus from our previous paper \cite{HLV19} in order to make the present text as self-contained as possible. 

\medskip
The category of $\Sy$-bimodules is endowed with a monoidal product: the \emph{connected composition product} $\boxtimes$ which amounts to composing operations along underlying 2-level connected  graphs, see \cite[Definition~1.13]{HLV19}. Its unit $\I$ is the $\Sy$-bimodule made up of a one-dimension space in arity $(1,1)$ spanned by $\id$. We will assume here that $\Sy$-bimodules are either \emph{left reduced},
 i.e. $\M(m,n)=0$ for $m=0$~, or \emph{right reduced}, i.e. $\M(m,n)=0$ for $n=0$~. 

\begin{definition}[Properads and coproperads]
A monoid in the monoidal category $(\Sbimod, \boxtimes, \I)$ is called a \emph{properad} and a comonoid is called a \emph{coproperad}.
\end{definition}

There exists a bar-cobar adjunction between 
augmented properads and conilpotent coproperads, see \cite[Section~2]{HLV19}. In the present paper, we will consider gebras over properads given by the cobar construction $\Cobar \calC$ of some coaugmented coproperad $\calC$. Such structures are equivalently encoded by Maurer--Cartan elements in the following convolution algebra, see \cite[Definition~2.20]{HLV19} for more details. 

\begin{definition}[Convolution algebra]
The \emph{convolution algebra} associated to a coaugmented dg coproperad $\calC$ and a dg vector space $A$ is defined by 
\begin{align*}
\g_{\calC, A}\coloneq \Big(\Hom_{\Sy}\big(\oC , \End_A\big)\cong \prod_{n,m\geqslant 0}
\Hom_{\Sy_m\times \Sy_n^{\mathrm{op}}}\big(\oC(m,n) , \Hom(A^{\otimes n}, A^{\otimes m})\big), \star, \partial\Big)\ ,
\end{align*}
with the Lie-admissible product given by 
\[f\star g \ : \ \oC \xrightarrow{\Delta_{(1,1)}} \oC \ibt  \oC \xrightarrow{f\ibt  g} \End_A \ibt  \End_A \xrightarrow{\gamma_{(1,1)}} \End_A~, \] 
where the first map $\Delta_{(1,1)}$ stands for the \emph{infinitesimal decomposition map} of the coaugmented  coproperad $\mathcal{C}$ which amounts to splitting its non-trivial elements into two. 
\end{definition}

Recall from \cite[Definition~3.12]{HLV19} that the \emph{left} and \emph{right infinitesimal composition products}
\[\M  \lhd_{(*)}  \N \qquad \text{and} \qquad \M 	\tensor*[_{(*)}]{\rhd}{}  \N\ , \] 
are the sub-$\Sy$-bimodules of $(\I \oplus \M)\boxtimes \N$ and $\M\boxtimes (\I \oplus \N)$ made up of the linear parts in $\M$ and $\N$ respectively. For $n\geqslant 1$, we consider their following summands 
\[
	\M \lhd_{(n)} \N \ 
	\qquad \text{and} \qquad\M 
	\tensor*[_{(n)}]{\rhd}{}
	\N \ ,  
\] 
 made up of one element of $\M$ on the bottom (resp. of $\N$ at the top) and $n$ elements of $\N$ on the top (resp. of $\M$ a the bottom). Notice that the top (resp. bottom) level is saturated by elements of $\N$ (resp. $\M$) and that the bottom (resp. top) level contains one element of $\M$ (resp. $\N$) and possibly many copies of the identity element from $\I$. \begin{figure*}[h]
	\begin{tikzpicture}[scale=0.75]
	\draw[thick] (1.95,1) to[out=270,in=90] (3.5,-1) ;
	\draw[thick] (4,-1)-- (4,-2);
	\draw[thick] (3,-1)-- (3,-2);
	\draw[thick] (6,-1)-- (6,-2);
	\draw[thick] (1,1) to[out=270,in=90] (2,-2);
	\draw[thick] (0.5,2) to  (0.5,1);
	\draw[thick] (0,1) to  (0,-2);
	\draw[thick] (7.5,2) to[out=270,in=90] (5.95,1) to[out=270,in=90] (8,-1) to[out=270,in=90] (8,-2); 
	\draw[thick] (5,2) to  (5,1) to[out=270,in=90] (4.5,-1);
	\draw[draw=white,double=black,double distance=2*\pgflinewidth,thick] (9,2) to  (9,1) to (9,-2);
	\draw[thick] (1.5,2) to  (1.5,1);
	\draw[draw=white,double=black,double distance=2*\pgflinewidth,thick] (8,1) to[out=270,in=90] (5.5,-1); 
	\draw[thick] (5,-1) 	 -- (5,-2);	
	\draw[draw=white,double=black,double distance=2*\pgflinewidth,thick] (4,2) to[out=270,in=90] (4,1);
	\draw[draw=white,double=black,double distance=2*\pgflinewidth,thick]  (4,1) to[out=270,in=90] (1,-2);
	\draw[draw=white,double=black,double distance=2*\pgflinewidth,thick] (6.5,2) to[out=270,in=90] (8,1) ;
	\draw[fill=white] (-0.3,0.8) rectangle (2.3,1.2);
	\draw[fill=white] (3.7,0.8) rectangle (6.3,1.2); 
	\draw[fill=white] (7.7,0.8) rectangle (9.3,1.2); 
	\draw[fill=white] (2.7,-1.2) rectangle (6.3,-0.8);
	\draw (1,1) node {\small{$\nu_1$}};
	\draw (5,1) node {\small{$\nu_2$}};
	\draw (8.5,1) node {\small{$\nu_3$}};
	\draw (4.5,-1) node {\small{$\mu$}};
	\end{tikzpicture}
	\caption{An element of $\M \lhd_{(3)} \N $.}
\end{figure*}

Recall from \cite[Definition~3.13]{HLV19} that the \emph{left} and \emph{right infinitesimal decomposition maps} of 
a coaugmented coproperad $(\calC, \Delta, \varepsilon)$ 
are defined respectively by 
\[\begin{tikzcd}[column sep=large, row sep=tiny]
\Cop{}{(*)} \ : \ \oC \arrow[r,"\Delta"] & 
 \calC \boxtimes \calC \arrow[r,"(\eps; \id)\boxtimes \id"] & \oC \libt  \calC\ , \\
\Cop{(*)}{} \ : \ \oC \arrow[r,"\Delta"] & 
 \calC \boxtimes \calC \arrow[r,"\id \boxtimes (\eps; \id)"] & \calC \ribt  \oC\ ,
\end{tikzcd}
\]
which can be extended to $\calC$ by setting the image of $\I$ to be trivial.
Similarly, they are made up of  the following components, for $n\geqslant 1$:
\[
	\begin{tikzcd}[column sep=normal, row sep=tiny]
		\Cop{}{(n)} \ : \ \calC \arrow[r,"\Delta"] & 
		\calC \boxtimes \calC \arrow[r, two heads] &\oC \lhd_{(n)} \calC\ , \\
		\Cop{(n)}{} \ : \ \calC \arrow[r,"\Delta"] & 
		\calC \boxtimes \calC\arrow[r, two heads] &\calC \tensor*[_{(n)}]{\rhd}{} \oC\ .
	\end{tikzcd}
\]

%\bruno{pourquoi pas $\Hom^A_B$ ?}
Given $f\in \Hom_{\Sy}\big(\calC, \End^A_B\big)$,  $\alpha \in \Hom_{\Sy}(\calC, \End_A)$, and $\beta\in \Hom_{\Sy}(\calC, \End_B)$, the \emph{left action}  of $\beta$ on $f$ and the \emph{right action} of $\alpha$ on $f$ are defined respectively by 
	\[
	\begin{tikzcd}[column sep=normal, row sep=tiny]
	\beta \lhd f  \ar[r,phantom,":" description] &
	\oC \arrow[r,"\Cop{}{(*)}"] &  \oC \libt \calC 
	\arrow[r,"\beta\libt f"] & \End_B \libt \End^A_B \arrow[r] & \End^A_B\ ,
\\
& \I \arrow[r,"\cong"] &  \I \boxtimes \I 
	\arrow[r,"\beta\boxtimes f"] & \End_B \boxtimes \End^A_B \arrow[r] & \End^A_B\ ,
\\
	f \rhd \alpha  \ar[r,phantom,":" description] &
	\oC \arrow[r,"\Cop{(*)}{}"] &  \calC \ribt \oC 
	\arrow[r,"f\ribt \alpha"] & \End^A_B \ribt \End_A \arrow[r] & \End^A_B \ ,\\
	& \I \arrow[r,"\cong"] &  \I \boxtimes \I 
	\arrow[r,"f\boxtimes \alpha"] & \End^A_B \boxtimes \End_A \arrow[r] & \End^A_B\ ,
	\end{tikzcd}
	\]
where the rightmost arrows are given by the usual composition of functions.	

\begin{definition}[$\infty$-morphism]\label{def:infmor}
An \emph{$\infty$-morphism} $f \colon (A,\alpha) \rightsquigarrow (B, \beta)$ is a 
degree $0$  map 
$f : \calC \to  \End^A_B$ 
of $\Sy$-bimodules satisfying the equation 
\begin{equation}\label{eq:Morph}
\partial (f)=
f  \rhd \alpha - \beta \lhd f\ .
\end{equation}
The \emph{composite} of $\infty$-morphisms is defined by 
\[\begin{tikzcd}
g\circledcirc f \ : \ \calC \arrow[r,"\Delta"] & 
\calC \boxtimes \calC \arrow[r,"g\boxtimes f"]  & 
\End^B_C \boxtimes \End_B^A \arrow[r]  & 
\End_C^A \  .
\end{tikzcd}\]
\end{definition}

\begin{proposition}[{\cite[Section~3]{HLV19}}]
The $\Cobar\calC$-gebras equipped with their $\infty$-morphisms and the composite $\circledcirc$ form a category, denoted by 
$\infty\text{-}\catofgebras{\Cobar\souche}$. 
\end{proposition}

The \emph{first component} of an $\infty$-morphism $f$ is defined by 
\[\begin{tikzcd}
f_{(0)} \ : \ \I  \arrow[r,hook] &  \I \oplus \oC \cong \calC  \arrow[r,"f"] & \End^A_B 
\end{tikzcd}\ .\]

\begin{definition}[$\infty$-quasi-isomorphism and $\infty$-isotopy]
An \emph{$\infty$-quasi-isomorphism} is an $\infty$-morphism $f : (A, \alpha) \rightsquigarrow (B, \beta)$ whose first component $f_{(0)} : A \xrightarrow{\cong} B$ is a quasi-isomorphism. 
An \emph{$\infty$-isotopy} is an $\infty$-morphism $f : (A, \alpha) \rightsquigarrow (A, \alpha')$ whose first component $f_{(0)}=\id_A$ is the identity.
\end{definition}

In the same way, an \emph{$\infty$-monomorphism} (resp. \emph{$\infty$-epimorphism}) is an $\infty$-morphism  whose first component is a monomorphism (resp. epimorphism). A \emph{strict morphism} is an $\infty$-morphism whose components vanish except possibly for the first one. In this case, this first component commutes strictly with the operations of the respective $\Omega\calC$-gebra structures.

\subsection{Resolution of the 2-colored properad}\label{subsec:ResColProp}

For any coaugmented dg coproperad $\calC$, the 2-colored dg properad encoding two $\Cobar \calC$-gebra structures related by a strict morphism of $\Cobar \calC$-gebras is given by 
\[
\left(\Cobar \calC\right)_{\bullet \to \bullet}\coloneq 
\left(\G\left(s^{-1}\oC_0 \oplus \I^0_1 \oplus s^{-1}\oC_1  \right)/\left( R  \right), d_0+d_1 \right)\ , 
 \]
where the first and the third generating summands are isomorphic to $s^{-1}\oC$ with the inputs and the outputs of $\oC_0$  and $\oC_1$ colored respectively  by $0$ and $1$,
where the third summand $\I^0_1$  is isomorphic to $\I$ with input color $0$ and output color $1$. 
We denote by $i$ its generator. The relations are given by 
\[R\coloneq
\left\{
\vcenter{\hbox{\begin{tikzpicture}[xscale=0.85, yscale=0.6]
	\draw[thick] (1,1)-- (1,1.4);
	\draw[thick] (0.25,1)-- (0.25,1.4);
	\draw[thick] (1.75,1)-- (1.75,1.4);
	\draw[thick] (0.5,0)-- (0.5,-0.4);	
	\draw[thick] (1.5,0)-- (1.5,-0.4);		
	\draw[fill=white] (0,0) rectangle (2,1);
	\draw (1,0.5) node {$s^{-1}c_0$};
	\draw (0.5,-0.8) node {$i$};	
	\draw (1.5,-0.8) node {$i$};	
	\end{tikzpicture}}}
\ -\ 
\vcenter{\hbox{\begin{tikzpicture}[xscale=0.85, yscale=0.6]
	\draw[thick] (1,1)-- (1,1.4);
	\draw[thick] (0.25,1)-- (0.25,1.4);
	\draw[thick] (1.75,1)-- (1.75,1.4);
	\draw[thick] (0.5,0)-- (0.5,-0.4);	
	\draw[thick] (1.5,0)-- (1.5,-0.4);		
	\draw[fill=white] (0,0) rectangle (2,1);
	\draw (1,0.5) node {$s^{-1}c_1$};
	\draw (1,1.8) node {$i$};	
	\draw (0.25,1.8) node {$i$};	
	\draw (1.75,1.8) node {$i$};			
	\end{tikzpicture}}}
	\ , c\in \oC
\right\}\ .\]
The image of $i$ under the differential is equal to $0$. One can see that the ideal $(R)$ is stable under the differential $d_0+d_1$ of the cobar construction; so this 2-colored dg properad is well defined. \\

In order to resolve it, we consider the following 2-colored dg properad:
\[
\left(\Cobar \calC\right)_{\bullet \rightsquigarrow \bullet}\coloneq \left(\G\left(s^{-1}\oC_0 \oplus \calC^0_1 \oplus s^{-1}\oC_1  \right), \d \right)\ , 
 \]
where the second generating summand is isomorphic to $\calC$, 
the inputs colored by $0$ and the outputs colored by $1$. 
The differential $\d$ is equal to the differential $d_0+d_1$ of the cobar construction on the first and the third summands of the generating space. The image of the second summand of the generating space under the differential $\d$ 
 is given by 
\begin{align*}
	d_{\calC^0_1} &+ 
	\left(\calC^0_1 \xrightarrow{\Cop{(*)}{}}  
	\calC^0_1\tensor*[_{(*)}]{\rhd}{} \oC_0 \xrightarrow{\id \tensor*[_{(*)}]{\rhd}{} s^{-1}}
	\calC^0_1\tensor*[_{(*)}]{\rhd}{} s^{-1}\oC_0
	\right)\\
	&- 
	\left(\calC^0_1 \xrightarrow{\Cop{}{(*)}}  
	\oC_1\lhd_{(*)} \calC^0_1 \xrightarrow{s^{-1}\lhd_{(*)} \id}
	s^{-1}\oC_1\lhd_{(*)}\calC^0_1
	\right).
\end{align*}

\begin{lemma}\label{lem:d2=0}
The map $\d$ squares to zero.
\end{lemma}

\begin{proof}
Since the map $\d$ is equal to the differential of the cobar construction on the first and the third summands, it is enough to check this on the second summand. The image of $\calC^0_1$ under $\d$ is made up of the following terms. There is first: $d_{\calC_1^0}^2\big(\calC^0_1\big)=0$. Let us denote by $(c_1,\ldots, c_n)\tensor*[_{(n)}]{\rhd}{} e$ the image of an element  $c\in 
\calC^0_1$ under $\Cop{(n)}{}$ and let us denote simply by $d$ the differential induced 
on $\G\left(s^{-1}\oC_0 \oplus \calC^0_1 \oplus s^{-1}\oC_1  \right)$
by the internal differential $d_\calC$ of $\calC$. We have
\begin{align*}
\left(\big(\id \tensor*[_{(n)}]{\rhd}{} s^{-1}\big)\circ  \Cop{(n)}{}\right) \circ d(c) = &  
-\sum_{i=1}^n (-1)^{|c_i|+\cdots+|c_n|}(c_1,\ldots, , d(c_i), \ldots, c_n)\tensor*[_{(n)}]{\rhd}{} (s^{-1}e)
\\
&+(c_1,\ldots, , d(c_i), \ldots, c_n)\tensor*[_{(n)}]{\rhd}{} (s^{-1}e)\\
=&-d\circ \left(\big(\id \tensor*[_{(n)}]{\rhd}{} s^{-1}\big)\circ  \Cop{(n)}{}\right)(c)\ .
\end{align*}

\medskip

A similar computation holds with $\Cop{}{(n)}$\ . In the end, all the terms of $\d^2(c)$ involving $d$ cancel.
The image of $c\in \calC^0_1$ under $\big(\id \tensor*[_{(*)}]{\rhd}{} s^{-1}\big)\circ  \Cop{(n)}{}$ followed by $d_0$ at the top vertex is equal to 
\[
-(-1)^{|e'|}(c_1, \ldots, c_n) \tensor*[_{(n)}]{\rhd}{} \left(s^{-1}e' \ibt s^{-1}e''\right)\ , 
\]
where $\Delta_{(1,1)}(e)\coloneq e' \ibt e''$ using Sweedler type notations. Recall that $d_2\big(s^{-1}e\big)=-(-1)^{|e'|}\big( s^{-1}e' \ibt s^{-1}e''\big)$.
By the coassociativity of the decomposition coproduct of the coproperad $\calC$, such a term also appears, but with the sign $(-1)^{|e'|}$  when one iterates $\big(\id \tensor*[_{(*)}]{\rhd}{} s^{-1}\big)\circ  \Cop{(n)}{}$ twice. The other terms produced by iterating $\big(\id \tensor*[_{(*)}]{\rhd}{} s^{-1}\big)\circ  \Cop{(n)}{}$ twice are of the form 
  
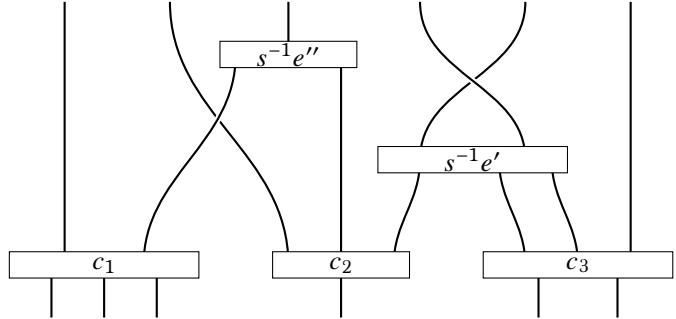
\[\begin{tikzpicture}[scale=0.7]
	\draw[thick] (11.5,6) -- (11.5,1) ;		
	\draw[thick] (6,5) -- (6,1) ;				
	\draw[thick] (5,6) -- (5,5) ;	
	\draw[thick] (0.5,1) -- (0.5,0) ;						
	\draw[thick] (1.5,1) -- (1.5,0) ;							
	\draw[thick] (2.5,1) -- (2.5,0) ;								
	\draw[thick] (6,1) -- (6,0) ;								
	\draw[thick] (9.75,1) -- (9.75,0) ;									
	\draw[thick] (11.25,1) -- (11.25,0) ;										
	\draw[thick] (0.75,6) -- (0.75,1) ;						
	\draw[thick] (9.5,6) to[out=270,in=90] (7.5,3);
	\draw[thick] (2.75,6) to[out=270,in=90] (5,1);	
	\draw[thick] (10,3) to[out=270,in=90] (10.5,1);		
	\draw[thick] (9,3) to[out=270,in=90] (9.5,1);			
	\draw[thick] (7.5,3) to[out=270,in=90] (7,1);			
	\draw[draw=white,double=black,double distance=2*\pgflinewidth,thick] (7.5,6)  to[out=270,in=90] (9.5,3); 
	\draw[draw=white,double=black,double distance=2*\pgflinewidth,thick] (4,5)  to[out=270,in=90] (2.25,1); 
	\draw[fill=white] (-0.3,0.75) rectangle (3.3,1.25);
	\draw[fill=white] (4.7,0.75) rectangle (7.3,1.25);
	\draw[fill=white] (8.7,0.75) rectangle (12.3,1.25);
	\draw[fill=white] (6.7,2.75) rectangle (10.3,3.25);
	\draw[fill=white] (3.7,4.75) rectangle (6.3,5.25);
	\draw (5,5) node {\small{$s^{-1} e''$}};
	\draw (8.5,3) node {\small{$s^{-1} e'$}};
	\draw (1.5,1) node {\small{$c_1$}};
	\draw (6,1) node {\small{$c_2$}};
	\draw (10.5,1) node {\small{$c_3$}};
\end{tikzpicture}\]
but each of them appears ``twice'' with different signs depending whether $e'$ or $e''$ comes first. So all these terms cancel. The same result holds true when one replaces $\Cop{(n)}{}$ by $\Cop{}{(n)}$. To conclude this part of the proof, it remains to show that the terms coming from $\Cop{(n)}{}$ followed by  $\Cop{}{(n)}$ are also produced 
by $\Cop{}{(n)}$ followed by  $\Cop{(n)}{}$ but with a different sign. In the first case, we get terms of the form 
\[
\left(c_1, \ldots, c_i, 
\left(s^{-1} f\right)\lhd_{(k)}(c_{i+1},\ldots, c_{i+k})  , 
c_{i+k+1}, \ldots, c_n
\right)
\tensor*[_{(n-k+1)}]{\rhd}{} \left(s^{-1} e\right)
\]
and, in the second case, we get terms of the form 
\[
\left(s^{-1} f\right)\lhd_{(n-k+1)}
\left(c_1, \ldots, c_i, 
(c_{i+1},\ldots, c_{i+k})\tensor*[_{(k)}]{\rhd}{} \left(s^{-1} e\right), 
c_{i+k+1}, \ldots, c_n\right) \ .
\]

\medskip

Using once again the coassociativity of the decomposition coproduct of the coproperad $\calC$, one can see that both types of term are equal: they are given by the summand of $\Delta^2(c)$ made up of 3-levels connected graphs made up of only one vertex labelled by $\oC$ on the top and on the bottom levels. Both terms come equipped with signs, but the second term above carries an extra minus sign due to the fact that one desuspends the bottom vertex labelled by $f$ before the top vertex labelled by $e$ and not the other way round. 
\end{proof}

\cref{lem:d2=0} ensures that we do get a  2-colored dg properad. A gebra structure over 
$\left(\Cobar \calC\right)_{\bullet \rightsquigarrow \bullet}$ amounts to the data of two $\Cobar \calC$-gebra structures related by an $\infty$-morphism by \cite[Proposition~3.16]{HLV19}.

\begin{proposition}\label{prop:ModelCatColoreddgProperads}
The category of 2-colored dg properads over a field $\k$ of characteristic $0$ admits a cofibrantly generated model category structure where 
weak equivalences and fibrations are given arity-wise and color-wise by quasi-isomorphisms and degree-wise epimorphisms respectively.
\end{proposition}

\begin{proof}
All the arguments of \cite[Theorem~1.1]{JY09}
for the category of 2-colored dg props hold as well in this setting. 
This result relies ultimately on the fact that the category of chain complexes  over a field $\k$ of characteristic $0$
admits a functorial path object given by 
\[\mathrm{Path}(A)\coloneq \Hom\left(\rmC_*\left(\Delta^1\right), A\right) \cong(A\oplus A \oplus s^{-1}A, \d)\ ,    \  \text{with} \ \ 
\d\big(a,a',s^{-1}a''\big)=\big(da, da', s^{-1}\big(a-a'-da''\big)\big)~,\]
where $\rmC_*\left(\Delta^1\right)$ stands for the cellular chains of the interval. 
\end{proof}

In order to get suitable homotopical properties for this 2-colored dg properad, we need to restrict ourselves to the following 
class of dg cooperads. Let us recall that a coaugmented coproperad is called \emph{conilpotent} \cite[Definition~2.19]{HLV19} when its decomposition map comes from its comonadic one: 
\[\widetilde{\Delta}\ :\ \oC \to \mathcal{G}^c\left(\oC\right) = \bigoplus_{\mathrm{g} \in \overline{\mathrm{G}}} \mathrm{g}\left(\mathcal{\oC}\right)~, \]
where $\overline{\mathrm{G}}$ stands for the set of reduced directed connected graphs.

\begin{definition}[Weight of a graph and coradical filtration]\label{def:weightcoradical}\leavevmode
\begin{itemize}
\item[$\diamond$] The \emph{weight} of a directed connected graph is equal to the number of its vertices. We denote 
by $\mathcal{G}^c\left(\oC\right)^{(\leqslant k)}$ the summand of $\mathcal{G}^c\left(\oC\right)$ supported by graphs of weight less or equal to $k$. 

\item[$\diamond$] The \emph{coradical filtration} of the coaugmentation coideal $\oC$ of a conilpotent coproperad is defined by: 
\[\scrR_k \oC \coloneq \left\{c \in \oC\ |\ \widetilde{\Delta}(c) \in \mathcal{G}^c\left(\oC\right)^{(\leqslant k)}
\right\}~. \]
\end{itemize}
\end{definition}

This means that an element $c$ of $\oC$ lives in $\scrR_k \oC$ when all the components of $\widetilde{\Delta}(c)$ are supported by graphs of weight less or equal to $k$. 
By definition, the coradical filtration is increasing and exhaustive:
\[
0=\scrR_{0} \oC\subset 
\scrR_1 \oC\subset \scrR_2 \oC\subset \cdots \subset \scrR_k \oC\subset \scrR_{k+1} \oC\subset \cdots 
\quad \& \quad 
\bigcup_{k \geqslant 1} \scrR_k \oC =\oC
\ .
\]

\begin{definition}[Conilpotent dg coproperad]\label{def:dgconil}
A \emph{conilpotent dg coproperad} is a dg coproperad $(\calC, \Delta, d)$ which is conilpotent and such that its differential lowers strictly the coradical filtration, that is 
\[d\left(\scrR_k \oC\right)\subset \scrR_{k-1} \oC~,\]
for any $k\geqslant 1$.
\end{definition} 

All the dg coproperads we will consider fit into this definition: 
the Koszul dual dg coproperad $\calP^{\ac}$ of a quadratic, possibly inhomogeneous, properad 
and the bar construction $\Bar \calP$ of an augmented properad without differential. 
(In the case of the bar construction of an augmented non-negatively graded dg properad, all the results below hold true but the arguments of the various proofs have to be modified including the filtration by the homological degree.) 

\begin{proposition}\label{prop:Res} 
Let $\calC$ be a conilpotent dg coproperad. 
The aforementioned  2-colored dg properad is a cofibrant resolution of the 2-colored dg properad encoding strict morphisms: 
\[
\left(\Cobar \calC\right)_{\bullet \rightsquigarrow \bullet} 
%= \left(\G\left(s^{-1}\oC_0 \oplus \calC^0_1 \oplus s^{-1}\oC_1 \right), \d \right)
\xrightarrow{\sim}
\left(\Cobar \calC\right)_{\bullet \to \bullet}
%= \left(\G\left(s^{-1}\oC_0  \oplus \I^0_1 \oplus s^{-1}\oC_1  \right)/\left( R  \right), d_0+d_1 \right)
\ .\]
\end{proposition}

\begin{proof}
Let us denote by $\rho$ the map from left to right defined by the identity on $s^{-1}\oC_0$ and $s^{-1}\oC_1$  and by the coaugmentation 
$\varepsilon : \calC^0_1\to  \I^0_1$ 
on $\calC^0_1$~. It commutes with the respective differentials. The only non-trivial case to check is for $c^0_1\in 
\overline{{\calC}^0_1}$, where 
\[
\rho\left(\d\left(c^0_1\right)\right)=
\varepsilon\left((d_\calC(c))^0_1\right)
\ +\ 
\vcenter{\hbox{\begin{tikzpicture}[xscale=0.75, yscale=0.5]
	\draw[thick] (1,1)-- (1,1.4);
	\draw[thick] (0.25,1)-- (0.25,1.4);
	\draw[thick] (1.75,1)-- (1.75,1.4);
	\draw[thick] (0.5,0)-- (0.5,-0.4);	
	\draw[thick] (1.5,0)-- (1.5,-0.4);		
	\draw[fill=white] (0,0) rectangle (2,1);
	\draw (1,0.5) node {$s^{-1}c_0$};
	\draw (0.5,-0.8) node {$i$};	
	\draw (1.5,-0.8) node {$i$};	
	\end{tikzpicture}}}
-\vcenter{\hbox{\begin{tikzpicture}[xscale=0.75, yscale=0.5]
	\draw[thick] (1,1)-- (1,1.4);
	\draw[thick] (0.25,1)-- (0.25,1.4);
	\draw[thick] (1.75,1)-- (1.75,1.4);
	\draw[thick] (0.5,0)-- (0.5,-0.4);	
	\draw[thick] (1.5,0)-- (1.5,-0.4);		
	\draw[fill=white] (0,0) rectangle (2,1);
	\draw (1,0.5) node {$s^{-1}c_1$};
	\draw (1,1.8) node {$i$};	
	\draw (0.25,1.8) node {$i$};	
	\draw (1.75,1.8) node {$i$};			
	\end{tikzpicture}}}
	=0 
	\ ,
\]
vanishes since $d_\calC(c)\in \oC=\ker \varepsilon$  and since  the sum of the last two terms lies in the ideal $(R)$. 

\medskip

Let us now show that the map $\rho$ is a quasi-isomorphism; to this extend, 
 we consider the filtrations $\scrF$ defined on both sides as follows. 
For any graph $g$ with vertices labelled by elements 
in  $s^{-1}\oC_0$, ${\calC}^0_1$, and $s^{-1}\oC_1$~, we consider only the 
labels $c^1, \ldots, c^l$ coming from $\oC_0$, $\overline{\calC}^0_1$, and $\oC_1$~. 
Since the underlying coproperad $\calC$ is conilpotent, its coradical filtration is exhaustive, so 
$c^i\in \scrR_{k_i} \oC$, for any $1\leqslant i \leqslant l$~. 
We declare that $g \in \scrF_k$ when $k_1+\cdots+k_l-l\leqslant k$~. 
By convention, we place $\I$ and $ \I^0_1$ in $\scrF_0$.
This defines an increasing and exhaustive filtration $0=\scrF_{-1}\subset \scrF_0\subset \scrF_1  \subset \cdots$ on both sides: it is enough to check that it is compatible with the relations $R$ on the right-hand side. 
It is straightforward to see that the morphism $\rho$ and the differentials preserve these filtrations:  the differentials $d_\calC$, $d_0$, and $d_1$ send $\scrF_k$ to $\scrF_{k-1}$ and  the 
differential $\d$ sends $\scrF_k$ to $\scrF_k$~. 

\medskip

On the first page $E^0$ of the associated spectral sequences, the right-hand side has no more differential and the only remaining part of the differential on the left-hand side amounts to 
\[
c^0_1 \mapsto  (i, \ldots, i) \tensor*[_{(m)}]{\rhd}{}
\left(s^{-1} c_0\right)- \left(s^{-1} c_1\right)  \lhd_{(n)} (i, \ldots, i) ~ ,\]
for any $c\in \calC(m,n)$~. 
This latter chain complex is isomorphic to the augmented two-sided cobar construction 
\[\Omega\calC \underset{\iota}{\boxtimes} \calC \underset{\iota}{\boxtimes} \Omega\calC\] of the coproperad defined on the underlying $\Sy$-bimodule $\calC$ equipped with the trivial decomposition map, that is where all the elements of $\oC$ are primitive, see \cite[Section~4.3]{Vallette07}. 
This coproperad is the Koszul dual coproperad of the free properad on $\oC$, which is Koszul. Therefore, if we consider the degree given by the number of elements coming from $\oC$ in the middle, the homology of this
augmented two-sided cobar construction is concentrated in degree $0$ where it is isomorphic to $\Omega \calC$~. This latter isomorphism is actually given by the identification $\Omega \calC\boxtimes \I \cong \I \boxtimes \Omega\calC$ where elements of $s^{-1}\oC$ on the left-hand side are pulled up from below the horizontal line made up of $\I$ to the area above it. 
This shows that $E^0(\rho)$ is a quasi-isomorphism.
Since the respective filtrations are complete (bounded below) and exhaustive, the Eilenberg--Moore comparison theorem 
\cite[Theorem~5.5.11]{WeibelBook} applies and shows that $\rho$ is a quasi-isomorphism. 

\medskip

It remains to prove that the left-hand 2-colored dg properad is cofibrant in the projective type model category given in \cref{prop:ModelCatColoreddgProperads}. 
The description of cofibrant dg properads in the projective type model category is given in \cite[Appendix~A.5]{MerkulovVallette09II}: they are retracts of quasi-free properads where the space of generators is equipped with a exhaustive increasing filtration lowered by $1$ by the differential, see Corollary~40 of \emph{loc. cit.}. The same arguments apply to the colored case, so here, we consider the increasing and exhaustive  filtration 
\[0= \scrF_{-1}\left( s^{-1}\oC_0 \oplus \calC^0_1 \oplus s^{-1}\oC_1\right)\subset 
\scrF_0\left( s^{-1}\oC_0 \oplus \calC^0_1 \oplus s^{-1}\oC_1\right) \subset  
\scrF_1\left( s^{-1}\oC_0 \oplus \calC^0_1 \oplus s^{-1}\oC_1\right)\subset \cdots \]
of the space of generators defined by 
\[
\scrF_k\left(s^{-1}\oC_0\right) \coloneq \scrR_{k+1}\left(s^{-1}\oC_0\right)~, \ 
\scrF_k\left(s^{-1}\oC_1\right) \coloneq \scrR_{k+1}\left(s^{-1}\oC_1\right)~, \ 
\scrF_k\left({\oC}^0_1\right) \coloneq \scrR_{k}\left({\oC}^0_1\right)~, \ \text{and}~, 
\]
for $k\geqslant 0$~, and $\scrF_{0}\left({\calC}^0_1\right)\coloneq \I^0_1$~.
Using the definition of the differential $\d$, it is straightforward to check that 
\[\d\left(\scrF_{0}\left( s^{-1}\oC_0 \oplus \calC^0_1 \oplus s^{-1}\oC_1\right)\right)=0
\quad \text{and}\]
and that 
\[
\d\left(\scrF_{k}\left( s^{-1}\oC_0 \oplus \calC^0_1 \oplus s^{-1}\oC_1\right)\right) \subset 
\G\left(\scrF_{k-1}\left( s^{-1}\oC_0 \oplus \calC^0_1 \oplus s^{-1}\oC_1\right)\right)~,  
\]
for any $k\geqslant 1$, 
which concludes the proof. 
\end{proof}

\begin{remark}
	When the coproperad $\calC$ is concentrated in arities $(1,n)$, that is when it is a cooperad, we recover the result of \cite[Lemma~55 and Proposition~56]{MerkulovVallette09I}.
\end{remark}

\cref{prop:Res} provides us with yet another homotopical justification for the definition of $\infty$-mor\-phi\-sms given in \cite{HLV19}.

\subsection{Characterisations of $\infty$-morphisms}
Forgetting the colors, we  now consider  the dg properad 
 \[\Cyl \left(\Cobar \calC \right)
  \coloneq
\left(\G\left(s^{-1}\oC\oplus \oC  \oplus s^{-1}\oC  \right), \d \right)\ , 
 \]
defined by the same kind of differential: the elements of $\I^0_1$ that were appearing previously when applying the differential are now replaced by the similar elements of the unit $\I$ of this quasi-free properad. This explains the slight abuse of notation for the differential  (same symbol). 
Notice that a $\Cyl \left(\Cobar \calC \right)$-gebra structure amounts to two $\Cobar \calC$-gebra structures, on the same underlying chain complex, related by an $\infty$-isotopy this time. 

\begin{proposition}\label{prop:Cylinder}
The dg properad $\Cyl \left(\Cobar \calC \right)$
is a functorial very good cylinder of the cobar construction $\Cobar \calC$ for  
conilpotent dg coproperads $\calC$. 
\end{proposition}

\begin{proof}
We consider the map 
\[\Phi \ : \ 
\Cobar \calC \vee \Cobar \calC \cong \left(\G\left(s^{-1}\oC \oplus s^{-1}\oC \right), d_0+d_1 \right) \to 
\Cyl \left(\Cobar \calC \right) =\left(\G\left(s^{-1}\oC \oplus \oC \oplus s^{-1}\oC  \right), \d \right)
\ ,\]
which sends identically the two generating summands $s^{-1}\oC$ to their respective counterparts on the right-hand side. We also 
 consider the map 
\[\Psi\ : \ \Cyl \left(\Cobar \calC \right) =\left(\G\left(s^{-1}\oC \oplus \oC\oplus s^{-1}\oC \right), \d \right) \to 
\Cobar \calC = \left(\G\left(s^{-1}\oC  \right), d \right)~, \]
which sends identically the two generating summands $s^{-1}\oC$ to the generating summand on the right-hand side 
and which sends the generating summand $\oC$  to zero. 
It is straightforward to check that they are morphisms of dg properads. Their composite 
\[ \Cobar \calC \vee \Cobar \calC \xrightarrow{\Phi} 
\Cyl \left(\Cobar \calC \right) \xrightarrow{\Psi}
\Cobar \calC
\]
is obviously equal to the fold map. 

\medskip

The exact same arguments as the ones used in of \cref{prop:Res} prove that the map $\Psi$ is a quasi-isomorphism. 
 The map $\Psi$ is always a fibration since it is a degreewise surjection in every arity. 
 
 \medskip

 To see that the map $\Phi$ is a cofibration, we use the characterisation given in \cite[Proposition~37]{MerkulovVallette09II}: 
 the map $\Phi$ embeds $\Cobar \calC \vee \Cobar \calC$ into $\Cobar \calC \vee \Cobar \calC \vee \G\big(\oC\big)$, where the extra generators in $\oC$ come equipped with the increasing and exhaustive coradical filtration 
 satisfying 
\[
\d\left(\scrR_k \oC\right) \subset 
\Cobar \calC \vee \Cobar \calC \vee 
\G\left(\scrR_{k-1}\oC\right)~,  
\]
for any $k\geqslant 1$~.  
(Notice that there is a typo in the published version of \emph{loc. cit.}: one should read $d : S_i \to \calP\vee \mathcal{F}(S_{i-1})$ instead of $d : S_i \to \mathcal{F}(S_{i-1})$.) The method to construct a retract from free $\Sy$-bimodules is given in   
\cite[Lemma~39]{MerkulovVallette09II}.

\end{proof}

This result shows that the deformation theory of morphisms of dg properads $\Omega \calC \to \End_A$ up to homotopy is equivalent to the deformation theory of $\Omega \calC$-gebra structures on $A$ up to $\infty$-isotopies. In other words, the problems (1) and (2) of \cite[Sections~6.2 \& 6.6]{GY19} are equivalent on the properadic level. 

\begin{remark}
	When the coproperad $\calC$ is concentrated in arities $(1,n)$, that is when it is a cooperad, we recover Theorem~5.1.4 of \cite{Fresse09ter} as claimed in Section~5.2.3 of \emph{loc. cit.}.
\end{remark}

This result gives another homotopical property satisfied by the notion of an $\infty$-isotopy.

\begin{proposition}\label{prop:InftyIsot}
Let $\calC$ be a conilpotent dg coproperad. 
If two $\Cobar \calC$-gebra structures $\alpha$ and $\alpha'$ on the same underlying chain complex $A$ are $\infty$-isotopic then they are related by a two-arrows zig-zag of (strict) quasi-isomorphisms of $\Cobar \calC$-gebras 
\[
\begin{array}{c}
\exists\  \text{$\infty$-isotopy} \\
\begin{tikzcd}[column sep=normal]
	(A,\alpha) \ar[r,squiggly,"="] 
	& (A, \alpha')
\end{tikzcd} 
\end{array}
\Longrightarrow
\begin{array}{c}
\exists\  \text{two-arrows zig-zag of quasi-isomorphisms} \\
\begin{tikzcd}
	(A, \alpha) 
	& \point \ar[l,"\sim"', "g"] \ar[r,"\sim", "h"'] 
	& (A, \alpha')
	\end{tikzcd}
  \ ,
\end{array}
\]
where the respective isomorphisms in homology satisfy $H(g)=H(h)^{-1}$~.
\end{proposition}

\begin{proof}
The assumption that the dg coproperad $\calC$ is conilpotent ensures that the dg properad $\Cobar\calC$ is cofibrant in the model category of dg properads \cite[Appendix~A]{MerkulovVallette09II} in the same way as in the proof of \cref{prop:Res}.
\cref{prop:Cylinder} shows that two $\Cobar \calC$-gebra structures $\alpha, \alpha'$ are $\infty$-isotopic if and only if 
their two associated structural maps $\Cobar \calC \to \End_A$ are (left) homotopic. 
The methods developed in \cite[Theorem~C \& Theorem~8.4]{Fresse10ter} on the level of props hold as well for properads, see Section~5.2.3 of \emph{loc. cit.}:  
since every chain complex over a field is fibrant and cofibrant, two homotopic $\Cobar \calC$-gebra structures $\alpha$ and $\alpha'$ on the same underlying chain complex $A$ fits into a zig-zag of (strict) quasi-isomorphisms 
\[\begin{tikzcd}
	(A, \alpha) 
	& (\mathrm{Path}(A), \varphi) \ar[l,"\sim"', "g"] \ar[r,"\sim", "h"'] 
	& (A, \alpha')
	\end{tikzcd}~, 
	\]
where the underlying quasi-isomorphisms $g$ and $h$ are the ones coming from the factorisation of the diagonal map of $A$ by the associated path object: 
\[\begin{tikzcd}
	A \ar[r,hook, "\sim"]  & 
	\mathrm{Path}(A)\cong(A\oplus A \oplus s^{-1}A, d)  \ar[r,two heads, "{(g,h)}"'] 
	& A \oplus A 
	\end{tikzcd}~.
	\]
Since these data satisfy 
\[d\big(a,a',s^{-1}a''\big)=\big(da, da', s^{-1}\big(a-a'-da''\big)\big)~,\quad  g\big(a, a',a''\big)=a~, \quad \text{and} \quad 
h\big(a, a',a''\big)=a'~,\]
we get the last claim, that is $H(g)=H(h)^{-1}$~. 
\end{proof}

\begin{remark}
The converse statement is obviously wrong. A quick counterexample is provided by any (associative) algebra structure $\alpha$ on a vector space $A$, that is with trivial differential, and by any isomorphism $f$ of the underlying space $A$ which fails to preserve $\alpha$. In this case, one can consider the transported associative algebra structure $\alpha'$ on $A$ under $f$. This defines a zig-zag of (quasi-)isomorphisms between $\alpha$ and $\alpha'$. But there cannot exist any $\infty$-isotopy $g$ from $\alpha$ to $\alpha'$ because, in this case, the first component $\left( g^{-1} \circledcirc f\right)_{(0)}=f$ would be an isomorphism of the algebra $\alpha$.
\end{remark}

In order to get a characterization of zig-zags of (strict) quasi-isomorphisms in terms of $\infty$-morphisms, one has to consider the weaker class of  
 $\infty$-quasi-isomorphisms.

\begin{theorem}\label{thm:MainInftyQi}
Let $\calC$ be a conilpotent dg  coproperad.
Two $\Cobar \calC$-gebra structures $(A,\alpha)$ and $(B, \beta)$ are $\infty$-quasi-isomorphic if and only if they are related by a zig-zag of (strict) quasi-isomorphisms of $\Cobar \calC$-gebras: 
\[
\begin{array}{c}
\exists\  \text{$\infty$-quasi-isomorphism} \\
\begin{tikzcd}[column sep=normal]
	(A,\alpha) \ar[r,squiggly,"\sim"] 
	& (B, \beta)
\end{tikzcd} 
\end{array}
\Longleftrightarrow
\begin{array}{c}
 \exists\  \text{zig-zag of quasi-isomorphisms} \\
\begin{tikzcd}[column sep=small]
	(A, \alpha) \ar[r,"\sim"] 
	& \point
	& \point \ar[l,"\sim"'] \ar[r,dotted, no head]
	&\point
	& \point \ar[l,"\sim"'] \ar[r,"\sim"] 
	& (B, \beta)
	\end{tikzcd}  \ .
\end{array}
\]
\end{theorem}

\begin{proof}\leavevmode
\begin{itemize}
\item[\sc ($\Longleftarrow$)] When two $\Cobar \calC$-gebra structures are 
 related by a zig-zag of (strict) quasi-isomorphisms, then there exists 
 a direct $\infty$-quasi-isomorphism between them, by the homological invertibility of $\infty$-quasi-isomorphisms
\cite[Theorem~4.18]{HLV19}.

\item[\sc ($\Longrightarrow$)] Let us first settle a preliminary result which gives a precise comparison between two ways of performing the homotopy transfer theorem. 
On one hand, notice that the methods and results developed in \cite[Section~7]{Fresse10ter} on the level of props, hold as well for properads. For instance, \cite[Proposition~7.4]{Fresse10ter} shows that for any quasi-isomorphism 
$f_{(0)} : C \stackrel{\sim}{\rightarrow} D$ and any $\Cobar \calC$-gebra structure $\delta$ on $D$, there exists
an $\Cobar \calC$-gebra structure $\gamma$ on $C$ and 
a zig-zag of quasi-isomorphisms 
\[
\begin{tikzcd}[column sep=small]
	(C, \gamma) \ar[r,"\sim", "i"'] 
	& (Z, \rho) 
	& (\mathrm{Path(Z)}, \varphi) \ar[r,"\sim", "h"'] \ar[l,"\sim"', "g"] 
	&(Z, \sigma) \ar[r,"\sim", "p"'] 
	& (D, \delta)\ ,
	\end{tikzcd} 
\]
where $\begin{tikzcd}[column sep=small]
	C \ar[r,"\sim", "i"'] 
	&Z \ar[r,"\sim", "p"'] 
	& D
	\end{tikzcd} $
is the acyclic fibration-cofibration factorisation of $f_{(0)}=pi$ and where
$g$ and $h$ are the two quasi-isomorphisms involved in the factorisation of the diagonal map of $Z$ by the associated path object, like in the proof of \cref{prop:InftyIsot}. By the same arguments as given above, one gets the following relation in homology:
\[
\H(p)\circ \H(h) \circ \H(g)^{-1}\circ \H(i)=\H(p) \circ \H(i)=\H\big(f_{(0)}\big)~.
\]
On the other hand, the other form of the homotopy transfer theorem given in \cite[Theorem~4.14]{HLV19} provides us with 
an $\Cobar \calC$-gebra structure $\gamma'$ on $C$ such that 
there exists an $\infty$-quasi-isomorphism 
${f}\colon (C,\gamma') \rightsquigarrow (D, \delta)$ whose first component is equal to 
$({f})_{(0)}=f_{(0)}$~. 
The homological invertibility of $\infty$-quasi-isomorphisms established in \cite[Theorem~4.18]{HLV19} shows that  there exists an $\infty$-quasi-isomorphism $\widetilde{f} \colon (C,\gamma) \rightsquigarrow (D, \delta)$ such that 
\[
\H\left(\widetilde{f}_{(0)}\right)=\H\big(f_{(0)}\big)~. 
\]
In the end, we get an $\infty$-quasi-isomorphism 
\[
g\coloneq \widetilde{f}^{-1} \circ f \colon (C, \gamma') 
\stackrel{\sim}{\rightsquigarrow} (C, \gamma)
\]
whose first component realises the identity on homology: $\H(g_{(0)})=\id_{\H(C)}$\ . 

\medskip

Let us now prove the claim of the statement. Suppose  that there exists an $\infty$-quasi-iso\-mor\-phi\-sm 
$f\colon (A,\alpha) \stackrel{\sim}{\rightsquigarrow} (B, \beta)$\ . We consider the homology of both underlying chain complexes: $H\coloneq \H(A) \cong \H(B)$\ . Since we work over a field $\k$, there exists a contraction from $A$ to $H$, which induces an $\Cobar \calC$-gebra structure $\theta$ on $H$ and an $\infty$-quasi-isomorphism $i \colon (H,\theta) \stackrel{\sim}{\rightsquigarrow} (A, \alpha)$, by  the homotopy transfer theorem \cite[Theorem~4.14]{HLV19}. Applying the above-mentioned results to the quasi-isomorphism $i_{(0)}$~, we get another $\Cobar \calC$-structure $\theta'$ on $H$, a zig-zag of (strict) quasi-isomorphisms from $(A, \alpha)$ to $(H, \theta')$, and an $\infty$-quasi-isomorphism 
$(H,\theta) \stackrel{=}{\rightsquigarrow} (H, \theta')$ whose first component realises the identity on homology, that is an 
$\infty$-isotopy. 
Then, we consider the composite 
$f \circledcirc i$ which provides us with an $\infty$-quasi-isomorphism $j \colon (H,\theta) \stackrel{\sim}{\rightsquigarrow} (B,\beta)$\ . 
We apply again the abovementioned result to the quasi-isomorphism $j_{(0)}$ this time: 
 it induces 
yet another $\Cobar \calC$-structure $\theta''$ on $H$, 
a zig-zig of quasi-isomorphisms from $(B, \beta)$ to $(H, \theta'')$, 
and an $\infty$-isotopy $(H,\theta) \stackrel{=}{\rightsquigarrow} (H, \theta'')$\ . In the end, the two $\Cobar \calC$-gebra structures $(H, \theta')$ and $(H, \theta'')$ are $\infty$-isotopic, so they are related by a zig-zag of quasi-isomorphisms by \cref{prop:InftyIsot}, which  concludes the proof. 
\[\begin{tikzcd}
		(A,\alpha)
		\ar[rr,squiggly,"\sim"]
			&
			& (B,\beta) 
		\\
		\point 
		\ar[u,"\rotatebox{90}{$\sim$}"]
		&& \point \ar[u,"\rotatebox{90}{$\sim$}"']
		\\
		\point 
		\ar[u,"\rotatebox{90}{$\sim$}"]
		\ar[d,"\rotatebox{90}{$\sim$}"']
			& (H,\theta)
			\ar[luu,squiggly,"\sim"']
			\ar[ruu,squiggly,"\sim"]
			\ar[ddl,squiggly,"="']
			\ar[ddr,squiggly,"="]
			&
			\point \ar[u,"\rotatebox{90}{$\sim$}"']\ar[d,"\rotatebox{90}{$\sim$}"]
		\\
		\point 
			&
			&\point 
		\\
		(H,\theta')\ar[u,"\rotatebox{90}{$\sim$}"]
		& \ar[l,"\sim"'] \point  \ar[r,"\sim"]&  (H,\theta'') \ar[u,"\rotatebox{90}{$\sim$}"']
			]
	\end{tikzcd}
\]
\end{itemize}
\end{proof}

\begin{remark}
When the coproperad $\calC$ is concentrated in arities $(1,n)$, that is when it is a cooperad, we recover 
the seminal property \cite[Theorem~11.4.9]{LodayVallette12} of $\infty$-quasi-isomorphisms between homotopy algebras. Notice that the known proofs of this property on the level of operads cannot be reproduced on the properadic level since they rely on the rectification of homotopy algebras \cite[Section~11.4.3]{LodayVallette12}. 
\end{remark}

\cref{thm:MainInftyQi} is the key result to be able to study efficiently the formality property for dg gebras over properads with methods from deformation theory. It is used in a crucial way in the preprint \cite{Emprin24} to settle \emph{Kaleidin--Emprin classes} which are faithful obstructions to the formality for dg gebras over properads and in the recent preprint \cite{EmprinTakeda25} to establish the coformality of spheres over the rational numbers with respect to a bialgebraic structure close to pre-Calabi--Yau algebras. 

\begin{proposition}\label{prop:PropertiesInfMor}
Let $\calC$ be a coaugmented dg coproperad.
\begin{enumerate}
\item For any $\infty$-monomorphism 
$\begin{tikzcd}
f \colon (A,\alpha) \ar[r,hook,squiggly]  &(B, \beta)
\end{tikzcd}$, there exists an $\infty$-isotopy 
$g \colon (B, \beta) \stackrel{=}{\rightsquigarrow}  (B, \beta')$, such that their composite is a (strict) morphism of $\Cobar \calC$-gebras equal to $g\cc f =f_{(0)} : (A, \alpha) \mono (B, \beta')$~.

\item For any $\infty$-epimorphism 
$\begin{tikzcd}
f \colon (A,\alpha) \ar[r,two heads,squiggly]  &(B, \beta)
\end{tikzcd}$, there exists an $\infty$-isotopy 
$g \colon (A, \alpha') \stackrel{=}{\rightsquigarrow}  (A, \alpha)$, such that their composite is a (strict) morphism of $\Cobar \calC$-gebras equal to $f\cc g =f_{(0)} : (A, \alpha') \epi (B, \beta)$\ .
\end{enumerate}
\end{proposition}

\begin{proof}\leavevmode
The arguments given in \cite[Proposition~4.4]{Vallette14} in the operadic case hold \emph{mutatis mutandis} on the level of properads. The only seminal change lies in the use of the forth characterisation of $\Cobar \calC$-algebras in the operadic Rosetta stone \cite[Theorem~10.1.13]{LodayVallette12} given in terms of codifferentials of cofree $\calC$-coalgebras. 
It is replaced here by the forth characterisation of $\Cobar \calC$-gebras in the properadic Rosetta stone  
\cite[Theorem~3.10]{HLV19} given in terms of bidifferentials of bifree monoid $\S \calC$-comodules. So, it is enough to change the elements of the form $\mathcal{P}^{\ac}(A)$ in the proof \cite[Proposition~4.4]{Vallette14} by 
$\S\calC \sq \S A$\ . 
\end{proof}

One might wish to extend the results of \cite[Section~4]{Vallette14} which state that $\infty$-quasi-isomorphisms, $\infty$-monomorphisms, and $\infty$-epimorphisms form almost a model category structure on $\Cobar \calC$-algebras. First, the existence of (finite) products fail for gebras over a properad. The axioms MC2 (two out of three) and MC3 (retracts) are straightforward. Due to the lack of products, the proof given in \emph{loc. cit.} of the axiom MC5b (factorisation cofibration-acyclic fibration) cannot be reproduced in the properadic setting. 

\begin{proposition}\label{prop:MC4-5}
Let $\calC$ be a conilpotent dg  coproperad.
\begin{description}
\item[\sc (MC4) Lifting property] Any commutative square of $\infty$-morphisms 
\[
\begin{tikzcd}[row sep=large, column sep=huge]
(A,\alpha) \ar[d,hook, squiggly] \ar[r, squiggly] & (B, \beta)\ar[d,two heads]  \\
(C, \gamma) \ar[r, squiggly] \ar[ur, squiggly] & (D, \delta)
\end{tikzcd}
\]
where either the left or the right vertical arrow is an $\infty$-quasi-isomorphism admits a diagonal $\infty$-morphism which splits it into two commuting triangles. 

\item[\sc (MC5a) Factorisation acyclic cofibration-fibration] Any $\infty$-morphism factor into the composite of an $\infty$-monomorphism which is also an $\infty$-quasi-isomorphism followed by an $\infty$-epimorphism: 
\[
\begin{tikzcd}
(A,\alpha) \ar[r,hook,"\sim", squiggly]  \ar[rr, bend right=20, squiggly]&
(C, \gamma) \ar[r,squiggly, two heads]& 
(B, \beta)\ .
\end{tikzcd}
\]
\end{description}
\end{proposition}

\begin{proof}
The arguments given in the proof of \cite[Theorem~4.2 (1)]{Vallette14} apply \emph{mutatis mutandis} 
with \cref{prop:PropertiesInfMor} instead of \cite[Proposition~4.4]{Vallette14}, 
with \cite[Section~5]{HLV19}, which actually works for 
conilpotent dg  coproperads, instead of \cite[Appendix~A]{Vallette14}, and 
with the space $\S\calC \sq \S B$ of \cite[Section~3.1]{HLV19} instead of $\mathcal{P}^{\ac}(B)$.
\end{proof}

\begin{remark}
In the proof of the abovementioned factorisation axiom (MC5a), the $\infty$-monomorphism\! 
\begin{tikzcd}
(A,\alpha) \ar[r,hook,"\sim"]  &
(C, \gamma) 
\end{tikzcd}\!\!
can be chosen to be strict. 
\end{remark}

\begin{proposition}\label{prop:PropertiesInfMorBIS}
When $\calC$ is a conilpotent dg  coproperad, any  $\infty$-morphism 
$f \colon (A,\alpha) \rightsquigarrow  (B, \beta)$ factorises into a composite of a monomorphism, an $\infty$-isotopy, and an epimorphism
\[
\begin{tikzcd}
(A,\alpha) \ar[r,hook,"\sim"]  \ar[rrr, bend right=15,"f"', squiggly]&
(C, \gamma) \ar[r,squiggly, "="]& 
(C, \gamma') \ar[r,two heads] & 
(B, \beta)\ ,
\end{tikzcd}
\]
where the monomorphism is a quasi-isomorphism.
\end{proposition}

\begin{proof}
We begin with the factorisation of the $\infty$-morphism $f$ into a monomorphism $i$, which is a quasi-isomorphism, followed by an $\infty$-epimorphism $q$~, given by (MC5a) of \cref{prop:MC4-5}.  Applying Point~(2) of \cref{prop:PropertiesInfMor} to $q$, there exists an $\infty$-isotopy 
\!\!\begin{tikzcd}
g \colon (C, \gamma')  \ar[r,squiggly, "="] & 
(C, \gamma)
\end{tikzcd}\!\!
such that the composite $p\coloneq q \circledcirc g$ is an epimorphism of $\Cobar \calC$-gebras. 
The final factorisation is given by 
\[
f= p \circledcirc g^{-1}\circledcirc i~.
\]
\end{proof}

\begin{corollary}
Let $\calC$ be a conilpotent dg  coproperad.
There exists an $\infty$-morphism between two  $\Cobar \calC$-gebra structures $(A,\alpha)$ and $(B, \beta)$  if and only if they are related by a zig-zag of (strict) morphisms of $\Cobar \calC$-gebras of the following form: 
\[
\begin{array}{c}
\exists\  \text{$\infty$-morphism} \\
\begin{tikzcd}[column sep=normal]
	(A,\alpha) \ar[r,squiggly] 
	& (B, \beta)
\end{tikzcd} 
\end{array}
\Longleftrightarrow
\begin{array}{c} 
\exists\  \text{zig-zag of morphisms} \\
\begin{tikzcd}
	(A, \alpha) \ar[r,hook, "\sim"] 
	& \point
	& \point \ar[l,"\sim"']\ar[r,"\sim"]
	& \point \ar[r,two heads] 
	& (B, \beta)
	\end{tikzcd}  \ ,
\end{array}
\]
where the monomorphism is a quasi-isomorphism.
\end{corollary}

\begin{proof}\leavevmode
\begin{itemize}
\item[\sc ($\Longleftarrow$)]
Any ($\infty$-)quasi-isomorphism of $\Cobar \calC$-gebras admits a ``homological'' inverse by \cite[Theorem~4.18]{HLV19}, then the composite of all the $\infty$-morphisms gives the desired direct $\infty$-morphism. 

\item[\sc ($\Longrightarrow$)]
This is a direct corollary of \cref{prop:PropertiesInfMorBIS} and \cref{prop:InftyIsot}. 
\end{itemize}
\end{proof}

\section{Homotopy theory with $\Linfty$-algebras}
In this section, we start by recalling the algebraic properties of (shifted complete curved) homotopy Lie algebras and with their (continuous) $\infty$-morphisms. Then, we introduce a universal $\Linfty$-algebra whose Maurer--Cartan elements encode pairs of 
$\Cobar \calC$-gebras related by an $\infty$-morphism. This construction allows us to enrich the category of $\Cobar \calC$-gebras and their $\infty$-morphisms with $\Linfty$-algebras and curved $\infty$-morphisms. 

\subsection{Homotopy Lie algebras}
This section recalls the main properties of complete shifted homotopy Lie algebras that will be used later on. 
We refer the reader to \cite[Section~1.1]{RNV19} and to \cite[Chapter~4]{DSV18} for more details. 

\medskip

A \emph{complete differential graded (dg) vector space} $(\g, d, \scrF)$ is a chain complex, i.e. 
a graded space $\g$ equipped with a degree $-1$ square-zero differential $d$, 
 endowed with a degree-wise filtration 
\[\g_n= \scrF_0 \g_n \supset \scrF_1 \g_n \supset \scrF_2 \g_n \supset \cdots \supset \scrF_k \g_n \supset \scrF_{k+1}\g_n \supset \cdots\]
made up of vector subspaces preserved by the differential $d(\scrF_k \g_n)\subset \scrF_k \g_n$~, such that the associated topology is complete. 
A morphism of complete dg vector spaces is a chain map preserving the respective filtrations. 
It forms a bicomplete closed symmetric monoidal category when equipped with the complete tensor product defined by $\g\widehat{\otimes}\h\coloneq \widehat{\g\otimes \h}$, where $\ \widehat{}\ $ stands for the completion functor. 
Its internal hom is denoted by $\hom$ and its complete filtration is defined by $$
\scrF_k \hom(\g, \h)\coloneq\left\{
f :\g \to \h\ | \ f(\scrF_n \g)\subset \scrF_{n+k} \h \ , \ \forall n\in \NN\right\}\ .$$

\begin{definition}[Complete shifted  $\L_\infty$-algebra]
A \emph{complete shifted  $\L_\infty$-algebra structure} $\g$, on a complete differential graded vector space $(\g, d, \scrF)$ 
satisfying $\g_n= \scrF_1 \g_n$ for any $n\in \ZZ$, 
is a collection 
$(\mathcal{l}_2, \mathcal{l}_3, \ldots)\in \prod_{n\geqslant 2} \hom\left(\g^{{\odot} n}, \g\right)_{-1}$ of symmetric maps of degree $-1$, which respect the filtration and which satisfy 
\[
 \partial\left(\mathcal{l}_n\right)+\sum_{p+q=n}
\sum_{\sigma\in \mathrm{Sh}_{p,q}^{-1}}
 (\mathcal{l}_{p+1}\circ_{1} \mathcal{l}_q)^{\sigma}=0\ ,
\]
for any $n\geqslant 2$, 
where $ \mathrm{Sh}_{p,q}^{-1}$ denotes the set of the inverses of $(p,q)$-shuffles.
\end{definition}

\begin{remark}
A shifted $\L_\infty$-algebra structure on a dg module $\g$ is equivalent to an $\L_\infty$-algebra structure on its desuspension $s^{-1}\g$~. 
\end{remark}

\begin{definition}[Maurer--Cartan element]
A \emph{Maurer--Cartan element} is a degree $0$ element $\alpha\in  \g_0$ of a complete shifted $\L_\infty$-algebras satisfying the Maurer--Cartan equation:
\[
d(\alpha)+\sum_{n\geqslant 2} {\textstyle \frac{1}{n!}}\mathcal{l}_n(\alpha, \ldots, \alpha)=0 \ ;
\]
we denote the associated set by $\MC(\g)$\ .
\end{definition}

Notice that this infinite series makes sense since we assume here $\g_0= \scrF_1 \g_0$~, so 
$\mathcal{l}_n(\alpha, \ldots, \alpha)\in \scrF_n \g_0$~, for $n\geqslant 2$.  

\begin{remark}
By convention, we might view the differential as the first structure operation, that is $\mathcal{l}_1\coloneq d$. This way both the definition of a complete shifted $\L_\infty$-algebra 
and the Maurer--Cartan equation  get more simple: 
\[
\sum_{p+q=n}
\sum_{\sigma\in \mathrm{Sh}_{p,q}^{-1}}
 (\mathcal{l}_{p+1}\circ_{1} \mathcal{l}_q)^{\sigma}=0
\quad \text{and} \quad \sum_{n\geqslant 1} {\textstyle \frac{1}{n!}}\mathcal{l}_n(\alpha, \ldots, \alpha)=0 \ .
\]
\end{remark}

\begin{proposition}[Twisting procedure]
Given any Maurer--Cartan element $\alpha$ of a complete shifted $\Linfty$-algebra $\g$, the following differential and operations 
\[
d^\alpha\coloneq \sum_{k\geqslant 0} {\textstyle \frac{1}{k!}} \mathcal{l}_{k+1}\big(\alpha^k, -\big) \qquad\text{and} \qquad 
\mathcal{l}_n^\alpha\coloneq \sum_{k\geqslant 0} {\textstyle \frac{1}{k!}} \mathcal{l}_{k+n}\big(\alpha^k, -,  \ldots,  -   \big)\ ,
\]
for $n\geqslant 2$, 
define another complete shifted $\Linfty$-algebra structure, whose Maurer--Cartan elements $\beta$ are in one-to-one correspondance 
with the Maurer--Cartan elements of the form $\alpha+\beta$ in $\g$~.
\end{proposition}

\begin{proof}
This can be proved by a straightforward computation. We refer to \cite[Section~4.4]{DSV18} for a more conceptual explanation. 
\end{proof}

This new structure is called the complete shifted $\Linfty$-algebra \emph{twisted} by the Maurer--Cartan element $\alpha$ and it is denoted by $\g^\alpha$. 

\subsection{Convolution homotopy Lie algebra}
Let $\calC$ be a coaugmented dg coproperad and let $A$ and $B$ be two dg vector spaces.
The goal of this section is to lift the dg Lie algebra structures on the convolution Lie-admissible algebras $\g_{\calC, A}$ and 
$\g_{\calC, B}$, recalled in \cref{subsec:PropCal}, to a shifted $\L_\infty$-algebra structure
on the dg vector space
\[s\Hom_{\Sy}\left(\oC, \End_A\right)\oplus \Hom_{\Sy}\left(\calC, \End^A_B\right) \oplus s\Hom_{\Sy}\left(\oC, \End_B\right)~, \]
so that it encodes $\infty$-morphisms of $\Omega \calC$-gebras. 
Let us first recall that the direct sum of two dg Lie algebras, with trivial bracket for mixed terms, is the categorical product of dg Lie algebras. Then, the suspension of a dg Lie algebra produces a shifted Lie algebra, that is a shifted $\L_\infty$-algebra where the structure operations of arity greater or equal to $3$ vanish. 
Here the shifted Lie bracket on $s\Hom_{\Sy}\left(\oC, \End_A\right)$  is explicitly given by
\[[s f, s g]_A=(-1)^{|f|}s (f\star g) - (-1)^{|f|(|g|+1)+1} s(g \star f) ~. \]

It remains to deal with the middle term. 
Given $n$ elements $f_1, \ldots, f_n$ of a graded vector space, we denote by 
\[
	f_1\odot \cdots \odot f_n \coloneq% \frac{1}{n!}
	\sum_{\sigma\in \Sy_n} \varepsilon_\sigma f_{\sigma(1)}\otimes \cdots \otimes f_{\sigma(n)}
\]
their symmetrization, where $\varepsilon_\sigma$ is the Koszul sign coming from the permutation of the graded terms. 
We consider the following multilinear operations acting on the total graded vector space:
\begin{eqnarray*}
&&\UU_{n+1}(s\beta, f_1, \ldots, f_n) 
	\ : \ 
	\calC \xrightarrow{\Cop{}{(n)}}  \oC \lhd_{(n)} \calC
	\xrightarrow{\beta\lhd_{(n)}(f_1\odot \cdots \odot f_n)} 
	\End_B \lhd_{(n)} \End^A_B
	\to 
	\End^A_B\ , \\
&&\DD_{n+1}(f_1, \ldots, f_n, s\alpha) 
	\ : \ 
	\calC \xrightarrow{\Cop{(n)}{}}  \calC\tensor*[_{(n)}]{\rhd}{} \oC
	\xrightarrow{(f_1\odot \cdots \odot f_n)\tensor*[_{(n)}]{\rhd}{} \alpha}
	\End^A_B  \tensor*[_{(n)}]{\rhd}{} \End_A
	\to 
	 \End^A_B\ , 
\end{eqnarray*}
for $\alpha\in \Hom_{\Sy}\left(\oC, \End_A\right)$, $\beta \in \Hom_{\Sy}\left(\oC, \End_B\right)$, and 
$f_1, \ldots, f_n\in \Hom_{\Sy}\left(\calC, \End^A_B\right)$, with $n\geqslant 1$~. 
We extend their definitions by symmetry: 
\begin{eqnarray*}
&&\UU_{n+1}(f_1, \ldots, f_k, s\beta, f_{k+1}, \ldots, f_n)\coloneq 
(-1)^{(|\beta|+1)(|f_1|+\cdots +|f_k|)}\, \UU_{n+1}(s\beta, f_1, \ldots, f_k, f_{k+1}, \ldots, f_n)~, \\
&&\DD_{n+1}(f_1, \ldots, f_k, s\alpha, f_{k+1}, \ldots, f_n)\coloneq 
(-1)^{(|\alpha|+1)(|f_{k+1}|+\cdots +|f_n|)}\, \DD_{n+1}(f_1, \ldots, f_k, f_{k+1}, \ldots, f_n, s\alpha)~.
\end{eqnarray*}
In the end, we define the structure operations by
\[	
	\left\{
	\begin{array}{ll}
		\mathcal{k}_2 \coloneq [\,\, ,\,]_A + [\,\, ,\,]_B+\DD_2-\UU_2\ ,
		& \text{for}\ n=2\ ,
		\\
		\mathcal{k}_n\coloneq \DD_n-\UU_n\ , 
		& \text{for}\  n\geqslant 3\ .
	\end{array}
	\right.
\]

\begin{theorem}\label{thm:ShiftedLinfini}
Let $\calC$ be a coaugmented dg coproperad and let $A$ and $B$ be two dg vector spaces. 
\begin{enumerate}
\item 
The data 
\[
	\mathfrak{k}_{\calC, A, B}\coloneq \left( 
	s\Hom_{\Sy}\left(\oC, \End_A\right)\oplus \Hom_{\Sy}\left(\calC, \End^A_B\right) \oplus s\Hom_{\Sy}\left(\oC, \End_B\right), 
	-\partial, \{\mathcal{k}_n\}_{n\geqslant 2}\right)
\]
forms a shifted $\L_\infty$-algebra structure.
\item This assignment is functorial as follows: given a morphism of coaugmented dg coproperads 
$G : \calD \to \calC$, the associated pullback morphism of graded vector spaces 
\begin{multline*}
G^* : 
s\Hom_{\Sy}\left(\oC, \End_A\right)\oplus \Hom_{\Sy}\left(\calC, \End^A_B\right) \oplus s\Hom_{\Sy}\left(\oC, \End_B\right)
\to  \\ 
s\Hom_{\Sy}\left(\oD, \End_A\right)\oplus \Hom_{\Sy}\left(\calD, \End^A_B\right) \oplus s\Hom_{\Sy}\left(\oD, \End_B\right)
\end{multline*}
 is a strict morphism of  shifted $\L_\infty$-algebras. 
\end{enumerate}
\end{theorem}

\begin{proof}\leavevmode
\begin{enumerate}
\item 
This is a corollary of \cref{lem:d2=0} using the general method of \cite[Theorem~28]{MerkulovVallette09I} as follows. The generating space 
$\oC_0 \oplus s\calC^0_1 \oplus \oC_1$ forms a 2-colored homotopy coproperad. So, the convolution space 
\[\Hom_{\Sy}\left(\oC_0 \oplus s\calC^0_1 \oplus \oC_1, \End_{A\oplus B}\right)
\cong 
\Hom_{\Sy}\!\left(\oC, \End_A\right) \oplus 
s^{-1}\Hom_{\Sy}\!\left(\calC, \End^A_B\right) 
\oplus 
\Hom_{\Sy}\!\left(\oC, \End_B\right) 
\]
 forms an $\L_\infty$-algebra. It remains to check that the formulas given in \cref{subsec:ResColProp} for this 
 2-colored homotopy coproperad structure on $\oC_0  \oplus s\calC^0_1\oplus \oC_1$ induce the abovementioned structure operations $\left\{\mathcal{k}_n\right\}_{n\geqslant 2}$~, which is straightforward. 

\item It is straightforward to see that $G^*$ commutes with the respective structure operations, since $G$ is a morphism of dg coaugmented coproperads and since the structure operations are defined with the respective decomposition maps. 
\end{enumerate}
\end{proof}

\begin{remark}
When the coproperad $\calC$ is concentrated in arities $(1,n)$, that is when it is a cooperad, 
the shifted $\L_\infty$-algebra $\mathfrak{k}_{\calC, A, B}$ coincides with the one given in 
\cite[Section~3.1]{DW14}. 
\end{remark}

It remains to define a complete filtration compatible with these structural operations. To this extend, we restrict ourselves to conilpotent dg coproperads $\calC$ introduced in \cref{def:dgconil}. Recall that we defined in \cref{def:weightcoradical} the weight of directed connected graph to be its number of vertices and that we considered there the induced coradical filtration on $\oC$ based on the weight. 

\begin{definition}[Size of a directed connected graph]
The \emph{size} a directed connected graph is equal to 
the maximal number of edges crossing any horizontal level. 
\end{definition}

\begin{exam}
The following directed graph has weight $4$ 
and size $5$. 
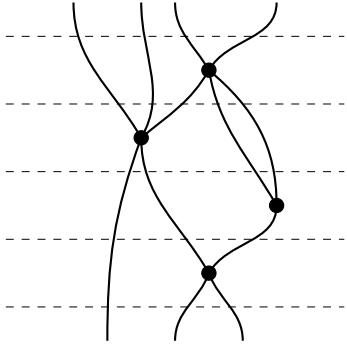
\[
\begin{tikzpicture}[scale=0.9,baseline=(n.base)]
\node (n) at (1,1.5) {}; %point base
\draw[dashed, very thin] (-3,3.5) -- (2,3.5);
\draw[dashed, very thin] (-3,2.5) -- (2,2.5);
\draw[dashed, very thin] (-3,1.5) -- (2,1.5);
\draw[dashed, very thin] (-3,0.5) -- (2,0.5);
\draw[dashed, very thin] (-3,-0.5) -- (2,-0.5);
\draw[fill=black] (0,0) circle (3pt); 
\draw[fill=black] (1,1) circle (3pt); 
\draw[fill=black] (-1,2) circle (3pt); 
\draw[fill=black] (0,3) circle (3pt); 
\draw[thick] (0,0) to[out=60,in=270] (1,1) to[out=120,in=280] (0,3) to[out=320,in=90] (1,1) ;
\draw[thick] (0,0) to[out=120,in=270] (-1,2) to[out=40,in=240] (0,3);
\draw[thick] (-1,2) to[out=60,in=270] (-1,4) node {};
\draw[thick] (-1,2) to[out=120,in=270] (-2,4) node {};
\draw[thick] (-1,2) to[out=250,in=90] (-1.5,-1) node {};
\draw[thick] (0,3) to[out=60,in=270] (1,4) node {};
\draw[thick] (0,3) to[out=120,in=270] (-0.5,4) node {};
\draw[thick] (0,0) to[out=250,in=90]  (-0.5,-1) node {};
\draw[thick] (0,0) to[out=290,in=90]  (0.5,-1) node {};
\end{tikzpicture}
\]
\end{exam}  

\begin{definition}[Density filtration]\label{def:DensOp}
Let $\calC$ be a conilpotent dg coproperad. 
 The \emph{density filtration} of a conilpotent dg coproperad $\calC$  is defined as follows :  an element $c \in \oC$ 
lives in $\scrD_k \calC$ if all the non-trivial components of $\widetilde{\Delta}(c)$ are supported by graphs whose sum of weight and size is less or equal to $k$. The identity element $\id$ belongs to  $\scrD_1 \calC$~. 
\end{definition}

 By definition, this filtration is increasing and exhaustive 
\begin{align*}
&0=\mathscr{D}_{0} \calC\subset 
\mathscr{D}_1 \calC\subset \mathscr{D}_2 \calC\subset \cdots \subset 
\bigcup_{k\geqslant 1} \mathscr{D}_k \calC =\calC 
~, 
\end{align*}
since the image $\widetilde{\Delta}(c)$ of any element $c\in \oC$ under the comonadic decomposition map is made up of a finite sum of labelled graphs. 
Since the comonadic decomposition map starts with the element $c$ itself labelling a one-vertex graph, i.e. 
$\widetilde{\Delta}(c)=c+\cdots$, 
the first non-trivial part of the coradical filtration is the space of primitive elements:
\[\scrR_1 \oC=\left\{c\in \oC \ |\ \widetilde{\Delta}(c)=c \right\}~.\]
Considering the top and the bottom levels of leaves of an element $c\in \calC(m,n)$, one sees that it cannot live in 
$\scrD_{\mathrm{max}(m,n)-1}$~. 

\begin{remark}
We could have equally well considered the filtration on conilpotent dg coproperads defined by the total number of leaves and internal edges of the directed connected graphs appearing in the image of the comonadic decomposition map $\widetilde{\Delta}$~. 
\end{remark}

\begin{definition}[Canonical filtration]
For any conilpotent dg coproperad $\calC$ and any dg vector spaces $A$ and $B$, 
the \emph{canonical filtration} of the graded vector space 
\[s\Hom_{\Sy}\left(\oC, \End_A\right)\oplus \Hom_{\Sy}\left(\calC, \End^A_B\right) \oplus s\Hom_{\Sy}\left(\oC, \End_B\right)~, \]
 is defined by 
\[	
	\left\{
	\begin{array}{ll}
s\alpha \in \scrF_k s\Hom_{\Sy}\left(\oC, \End_A\right) &  \text{if}\  \ \alpha|_{\mathscr{R}_{k-1} \oC}= 0~,
		\\
s\beta \in \scrF_k s\Hom_{\Sy}\left(\oC, \End_B\right) & \text{if}\ \ \beta|_{\mathscr{R}_{k-1} \oC}= 0~,
		\\
f \in \scrF_k \Hom_{\Sy}\left(\calC, \End^A_B\right) & \text{if}\ \ f|_{\mathscr{D}_{k-1} \calC}= 0~,
	\end{array}
	\right.
\]
\end{definition}

\begin{proposition}\label{prop:CompHoLieAlgeStructure}
Let $\calC$ be a conilpotent dg coproperad and let $A$ and $B$ be two dg vector spaces. 
The convolution $\L_\infty$-algebra  
\[
	\mathfrak{k}_{\calC, A, B}\coloneq \left( 
	s\Hom_{\Sy}\left(\oC, \End_A\right)\oplus \Hom_{\Sy}\left(\calC, \End^A_B\right) \oplus s\Hom_{\Sy}\left(\oC, \End_B\right), -\partial, \{\mathcal{k}_n\}_{n\geqslant 2}, \left\{\scrF_k\right\}_{k\geqslant 1}
	\right)
\]
is complete with respect to the canonical filtration. This assignment is functorial with respect to morphisms of dg conilpotent coproperads. 
\end{proposition}

\begin{proof}
By definition, the canonical filtration is decreasing and $\scrF_1$ is equal to the total space. 
It is complete since the coradical and the density filtrations are exhaustive.
Then, since the underlying differential $d_{\oC}$ of $\oC$ is a coderivation with respect to the graphs comonadic coalgebra structure, it preserves both the coradical and the density filtrations. This implies that the differential $\partial$ 
of the convolution $\L_\infty$-algebra 
preserves the canonical filtration.

Then, let us show that the structure operations of a convolution $\L_\infty$-algebra 
 preserve the canonical filtration. 
 We begin with the Lie-admissible product $\star$~. 
  Let $\alpha_1\in \scrF_{i_1} \Hom_{\Sy}\left(\oC, \End_A\right)$ and $\alpha_2 \in \scrF_{i_2} \Hom_{\Sy}\left(\oC, \End_A\right)$. 
  We have to show that $\alpha_1 \star \alpha_2 \in \scrF_{i_1+i_2} \Hom_{\Sy}\left(\oC, \End_A\right)$~, that is 
\[\left(\alpha_1\star \alpha_2 \right)|_{\mathscr{R}_{i_1+i_2-1}\oC}
=0~. \]
For any $c\in \scrR_{i_1+i_2-1}\oC$, we claim that all the components of its image under the infinitesimal decomposition map
$\Delta_{(1,1)}(c)=\sum \gra(c_1,c_2)$ (where the graphs  $\gra$ have $2$ vertices)
satisfy 
$c_1\in \scrR_{i_1-1}$ or $c_2\in \scrR_{i_2-1}$~; this would be enough to conclude since then one would have 
$\gra(\alpha_1(c_1), \alpha_2(c_2))=0$~. 
Suppose that this property is wrong, that is $c_1\notin \scrR_{i_1-1}$ and $c_2\notin \scrR_{i_2-1}$. 
This would mean that the image of $c_1$ (resp. $c_2$) under the comonadic decomposition map 
would contain a non-trivial component based on a graph $\gra_1$ (resp. $\gra_2$) with at least $i_1$ (resp. $i_2$)  vertices. 
The coassociativity of the graphs comonad would then show that there exists a non-trivial component 
based on the graph $\gra(\gra_1, \gra_2)$ with at least $i_1+i_2$ vertices in $\widetilde{\Delta}(c)$, which would contradict the fact that 
$c\in \scrR_{i_1+i_2-1}\oC$~.

Let us apply the same kind of arguments for the operations $\UU_n$ and $\DD_n$~. We will only write the proof for the first ones since it will be the same for the second ones. 
Let 
$\beta \in \scrF_{i_0} \Hom_{\Sy}\left(\oC, \End_B\right)$ 
and
$f_1\in \scrF_{i_1} \Hom_{\Sy}\left(\calC, \End^A_B\right), \ldots, f_n\in \scrF_{i_n} \Hom_{\Sy}\left(\calC, \End^A_B\right)$, that is 
$\beta|_{\mathscr{R}_{i_0-1} \oC}= 0$ and 
$f_{i_j}|_{\mathscr{D}_{i_j-1} \calC}= 0$, for $1\leqslant j \leqslant n$~. 
We have to show that $\UU_{n+1}(s\beta, f_1, \ldots, f_n)\in \scrF_{i_0+i_1+\cdots+i_n} \Hom_{\Sy}\left(\calC, \End^A_B\right)$~, that is 
$\UU_{n+1}(s\beta, f_1, \ldots, f_n)|_{\mathscr{D}_{i_0+i_1+\cdots+i_n-1}\calC}
=0$~. 
For any $c\in \mathscr{D}_{i_0+i_1+\cdots+i_n-1}\calC$, we claim that all the components of its image 
under the right infinitesimal decomposition map
$\Delta_{(n)}(c)=\sum \gra(c_0, c_1, \ldots,  c_n)$ (where the graphs $\gra$ have  $2$ levels with bottom level made up of one vertex labelled by $c_0$ and unlabelled egdes and top level made up of $n$ vertices or edges labelled by $c_1, \ldots, c_n$~)
satisfy 
$c_0\in \scrR_{i_0-1}$ or there exists $1\leqslant j \leqslant n$  such that $c_j\in \scrD_{{i_j}-1}$~; this would be enough to conclude since then one would have 
$\beta(c_0)\otimes f_1(c_1)\otimes \cdots \otimes f_n(c_n)=0$~. 
Suppose that this property is wrong, that is 
$c_0\notin \scrR_{i_0-1}$ and 
$c_j\notin \scrD_{{i_j}-1}$~, for all $1\leqslant j \leqslant n$~. 
This would mean that the image of $c_0$ under the comonadic decomposition map 
would contain a non-trivial component based on a graph $\gra_0$ with at least $i_0$ vertices and that 
 the image of $c_j$ under the comonadic decomposition map 
would contain a non-trivial component based on a graph $\gra_j$ with at least $i_j$ vertices plus crossings of edges along one level, for any $1\leqslant j \leqslant n$~. 
(If $c_j=\id$, then $c_j\notin \scrD_{{i_j}-1}$ implies $i_j=1$ and the arguments still hold with the trivial graph which has no vertex and whose size is equal to $1$.) 
Let us denote by $v_j$ its number of vertices and $l_j$ its maximum number of edges crossing one level, so that $v_j+l_j\geqslant i_j$~. 
The coassociativity of the graphs comonad shows that  
the graph $\gra(\gra_0, \gra_1, \ldots, \gra_n)$ supports a non-trivial component of $\widetilde{\Delta}(c)$~. 
It admits at least $i_0+v_1+\cdots+v_n$ vertices and at least $l_1+\cdots+l_n$ edges crossing one level, since the graphs $\gra_1, \ldots, \gra_n$ are put horizontally one next to each other in $\gra$~. This shows that the sum of the weight and the size of the graph $\gra(\gra_0, \gra_1, \ldots, \gra_n)$ is least $i_0+i_1+\cdots+i_n$~, which contradicts the fact that 
$c\in \scrD_{i_0+i_1+\cdots+i_n-1}\calC$~.

The functoriality is a straightforward consequence of Point~(2) of \cref{thm:ShiftedLinfini} and the definitions of the various filtrations. 
\end{proof}

\begin{theorem}\label{prop:MCInfty}
A triple $(\alpha, f, \beta)$ is a Maurer--Cartan element in the  shifted $\L_\infty$-algebra 
$\mathfrak{k}_{\calC, A, B}$ if and only if 
$(A,\alpha)$ and $(B, \beta)$ are two $\Omega \calC$-gebras and 
$f$ is an $\infty$-morphism between them. 
\end{theorem}

\begin{proof}
The part of the proof dealing with $\alpha$ and $\beta$ is already known, see \cite[Proposition~17]{MerkulovVallette09I}. 
Regarding the term $f$, it is enough to see that the remaining summand of the Maurer--Cartan equation, involving at least one of it, is equal to
\begin{align*}
-\partial(f) + \sum_{n\geqslant 2}  \tfrac{1}{(n-1)!} \DD_n(f , \ldots, f, s\alpha) -  \sum_{n\geqslant 2}  \tfrac{1}{(n-1)!}\UU_n(s\beta,f , \ldots, f) = 
-\partial(f)+f \lhd \alpha - \beta \rhd f=0\ , 
\end{align*}
which is  the equation of $\infty$-morphisms, under the notations of \cref{subsec:PropCal}.
\end{proof}

\begin{corollary}\label{cor:HoLieAlgeStructureBIS}
Let $\calC$ be a coaugmented coproperad and let $(A,\alpha)$ and $(B, \beta)$ be two $\Cobar\calC$-gebra structures. 
\begin{enumerate}
\item The operations 
\[	
	\left\{
	\begin{array}{ll}
		\mathcal{l}_1 \coloneq \DD_2(-,s\alpha)-\UU_2(s\beta, -)- \partial\ ,
		& \text{for}\ n=1\ ,
		\\
		\mathcal{l}_n\coloneq \DD_{n+1}(-, \cdots, -, s\alpha)-\UU_{n+1}(s\beta, -, \cdots, -) \ , 
		& \text{for}\  n\geqslant 2\ ,
	\end{array}
	\right.
\]
endow the graded vector space $\Hom_{\Sy}\big(\calC, \End^A_B\big)$
with a  shifted $\L_\infty$-algebra structure
\[
	\mathfrak{h}_{\alpha,\beta}\coloneq \left(\Hom_{\Sy}\big(\calC, \End^A_B\big), \{\mathcal{l}_n\}_{n\geqslant 1}\right)~.
\]

\item When the dg coproperad $\calC$ is conilpotent, the canonical filtration $\left\{\scrF_k \Hom_{\Sy}\big(\calC, \End^A_B\big)\right\}_{k\geqslant 1}$ endows it with a complete shifted $\L_\infty$-algebra structure. 

\item This assignment is functorial as follows: given a morphism of conilpotent dg coproperads 
$G : \calD \to \calC$, the morphism of filtered graded vector spaces 
$G^* : \Hom_{\Sy}\big(\calC, \End^A_B\big) \to \Hom_{\Sy}\big(\calD, \End^A_B\big)$ is a strict morphism of complete shifted $\L_\infty$-algebras when one considers the $\Cobar \calD$-gebra structures $G^*(\alpha)$ and $G^*(\beta)$ on $A$ and $B$ respectively.
\end{enumerate}
\end{corollary}

\begin{proof}\leavevmode
\begin{enumerate}
\item
This can be seen as a corollary of Point~(1) of \cref{thm:ShiftedLinfini}: the shifted $\L_\infty$-algebra $\mathfrak{h}_{\alpha,\beta}$ is equal to the subalgebra supported by 
$\Hom_{\Sy}\big(\calC, \End^A_B\big)$ 
of the complete shifted 
$\L_\infty$-algebra $\mathfrak{k}_{\calC, A, B}$ twisted by the Maurer--Cartan element $\alpha+\beta$.

\item  This is a corollary of \cref{prop:CompHoLieAlgeStructure} under the abovementioned interpretation with the twisting procedure. Let us notice here that, one could consider as well the canonical filtration on $\Hom_{\Sy}\big(\calC, \End^A_B\big)$ induced by the filtration on $\calC$ defined by the size only. 

\item This is a corollary of  Point~(2) of \cref{thm:ShiftedLinfini} under the abovementioned interpretation with the twisting procedure. 

\end{enumerate}
\end{proof}

\begin{corollary}\label{prop:MCInfMor}
	The set of Maurer--Cartan elements of the shifted homotopy Lie algebra $\mathfrak{h}_{\alpha,\beta}$ is equal to the set of $\infty$-morphisms between the $\Cobar \calC$-gebras $(A,\alpha)$ to $(B, \beta)$.
\end{corollary}

\begin{proof}
With the interpretation of the complete shifted $\L_\infty$-algebra $\mathfrak{h}_{\alpha,\beta}$ as a subalgebra of the complete shifted 
$\L_\infty$-algebra $\mathfrak{k}_{\calC, A, B}$ twisted at $\alpha+\beta$, 
this is a direct corollary of \cref{prop:MCInfty}. 
\end{proof}

\begin{remark}
	When the coproperad $\calC$ is concentrated in arities $(1,n)$, that is when it is a cooperad, we recover the result of \cite[Theorem~57]{MerkulovVallette09I}. 
\end{remark}

\begin{convention}
Since we will always work in the complete setting and under the shifted convention throughout the rest of this paper, \underline{we will  drop the articles ``complete'' and ``shifted'' for simplicity.}
\end{convention}

\subsection{Curved $\infty$-morphisms} We will work with  the following generalisation of the notion of an $\Linfty$-algebra.

\begin{definition}[Curved  $\L_\infty$-algebra]
A \emph{curved  $\L_\infty$-algebra structure}, on a complete graded module $(\g, \scrF)$ satisfying $\g_n= \scrF_1 \g_n$ for any $n\in \ZZ$, 
 is a collection 
$(\mathcal{l}_0, \mathcal{l}_1, \mathcal{l}_2, \ldots)\in \prod_{n\geqslant 0} \hom\left(\g^{{\odot} n}, \g\right)_{-1}$ of symmetric maps of degree $-1$, which respect the filtration and which satisfy, for any $n\geqslant 0$, 
\[
\sum_{p+q=n}
\sum_{\sigma\in \mathrm{Sh}_{p,q}^{-1}}
 (\mathcal{l}_{p+1}\circ_{1} \mathcal{l}_q)^{\sigma}=0\ .
\]
\end{definition}

An $\Linfty$-algebra is a curved $\Linfty$-algebra with trivial curvature $\mathcal{l}_0=0$. 

\begin{remark}
Recall that an $\Linfty$-algebra is an algebra over the operad $\Cobar \Com^*$, that is a Maurer--Cartan element $\mathcal{l}$ in the convolution pre-Lie algebra 
\[
      \left(\hom_\Sy(\Com^*,\eend_A)\cong \prod_{n\geqslant 1}\hom\left(\g^{{\odot} n}, \g\right),\partial,  \star\right)\ .
\]
We denote by $\mathrm{c}_n$ the generating element of $\Com^*(n)$~; under this convention, we have 
$\mathcal{l}(\mathrm{c}_n)=\mathcal{l}_n$. 
The operad $\Com$ encoding commutative algebras admits an extension with the operad 
$\uCom$ encoding unital commutative algebras. 
Maurer--Cartan elements in the convolution pre-Lie algebra 
\[
      \left(\hom_\Sy(\uCom^*,\eend_A)\cong \prod_{n\geqslant 0}\hom\left(\g^{{\odot} n}, \g\right),\partial,  \star\right)\ .
\]
associated to the (partial) cooperad $\uCom^*$ are in one-to-one correspondence with  curved $\Linfty$-algebra structures on $A$, see \cite[Chapter~3]{DSV18}. 
\end{remark}

\begin{definition}[$\infty$-morphism of curved $\Linfty$-algebras]
An \emph{$\infty$-morphism} between two curved $\Linfty$-algebras $(\g,\mathcal{l})$ and $(\h,\mathcal{k})$ is  
a collection 
$(\Phi_0, \Phi_1, \Phi_2, \ldots)\in \prod_{n\geqslant 0} \hom\left(\g^{{\odot} n}, \h\right)_0$ of symmetric maps of degree $0$ 
satisfying 
\begin{equation}\label{eq:CurvedLinfty}
\sum_{p+q=n}
\sum_{\sigma\in \mathrm{Sh}_{p,q}^{-1}}
 (\Phi_{p+1}\circ_{1} \mathcal{l}_q)^{\sigma}=
 \sum_{i_1+\cdots+i_k=n}
\sum_{\sigma\in \mathrm{Sh}_{i_1, \ldots, i_k}^{-1}}
 \mathcal{k}_k\circ \left(\Phi_{i_1}, \ldots, \Phi_{i_k}\right)^{\sigma}\ ,
\end{equation}
for any $n\geqslant 0$. It is called \emph{continuous} when the maps respect the filtrations. 
The composite of curved $\infty$-morphisms is given by 
\begin{equation}
(\Psi\circledcirc \Phi)_n\coloneq 
 \sum_{i_1+\cdots+i_k=n}
 \sum_{\sigma\in \mathrm{Sh}_{i_1, \ldots, i_k}^{-1}}
\Psi_k\circ \left(\Phi_{i_1}, \ldots, \Phi_{i_k}\right)^{\sigma}
\ .
\end{equation}
We denote $\infty$-morphisms of  curved $\Linfty$-algebras by $\Phi\colon \frakg \rightsquigarrow \frakh$ and the associated category by $\cLi$~.
The subcategory made up of continuous $\infty$-morphisms is denoted by 
$\mathsf{cont}\textsf{-}\mathsf{cL}_\infty$~.
\end{definition}

\begin{remark}
This notion can be obtained by applying \cref{def:infmor} to the ``cooperad'' $\uCom^*$ seen as a ''coproperad'' concentrated in arity $(1,n)$, for $n\geqslant 0$.
Under the notation $\Phi_n\coloneq \Phi(\mathrm{c_n})\colon \g^{{\odot} n} \to \frakh$, \cref{eq:CurvedLinfty} given above coincides with \cref{eq:Morph}, which is 
$\Phi\rhd  \mathcal{l}     =  \mathcal{k} \lhd \Phi$~. 
Strictly speaking, the linear dual $\uCom^*$ of the operad $\uCom$ does not form a cooperad with full decomposition map, but only a cooperad with partial decomposition maps. The  issue here is that the operator $\Cop{}{(*)}$ might produce infinite series. The property that  $\Phi(\mathrm{c}_0)(1)=\Phi_0(1)\in \scrF_1\h$, for $1\in \k\cong \g^{\odot 0}$, ensures that all the formulae actually make sense.
\end{remark}
 
Notice that a (curved) $\infty$-morphism between two $\Linfty$-algebras with trivial constant component $\Phi_0=0$ is a (classical) $\infty$-morphism, see \cite[Proposition~10.2.7]{LodayVallette12}. The data of an $\infty$-morphism 
$\Phi=(\Phi_0, \Phi_1, \Phi_2, \ldots)\colon \frakg \rightsquigarrow \frakh$ between two curved $\Linfty$-algebras is equivalent to a Maurer--Cartan element $a\coloneq \Phi_0(1)\in \MC(\h)$ and an $\infty$-morphism of the form $(0, \Phi_1, \Phi_2, \ldots)\colon \g \rightsquigarrow \h^a$ from $\g$ to $\h$ twisted by $a$, see \cite[Chapter~3, Section~4]{DSV18} for more details. The purpose for considering this enhanced version of $\infty$-morphisms due to Dolgushev--Rogers \cite{DolgushevRogers17} is to model non-necessarily pointed maps between Maurer--Cartan space, see \cref{subsec:Integration}. 

\medskip 

The Maurer--Cartan assignment 
can be improved into a functor $\MC \colon \mathsf{cont}\textsf{-}\mathsf{cL}_\infty \to \mathsf{Set}$ under the formula 
\[\MC(\Phi)(\alpha)\coloneq \sum_{n\geqslant 0}{\textstyle \frac{1}{n!}}\Phi_n(\alpha, \ldots, \alpha)~,\]
where $\Phi_0()=\Phi_0(1)$ by a slight abuse of notation. Such a formula always makes sense for continuous $\infty$-morphisms, otherwise one has to check directly that it is well defined, see \cref{thm:EnrichCat} for instance. 

\medskip

The direct sum $(\g,\mathcal{l}) \oplus (\h,\mathcal{k})$ of  curved $\Linfty$-algebras is defined by 
\[\mathcal{m}_n(g_1\oplus h_1, \ldots, g_n\oplus h_n)\coloneq 
\mathcal{l}_n(g_1, \ldots, g_n)+
\mathcal{k}_n(h_1, \ldots, h_n)
\ . \] 
It is functorial with respect to the following data: given two $\infty$-morphisms 
$\Phi\colon \frakg \rightsquigarrow \frakg'$ and $\Psi \colon \frakh \rightsquigarrow \frakh'$, we consider the 
$\infty$-morphism 
$\Upsilon\colon \frakg \oplus \frakg' \rightsquigarrow \frakh\oplus \frakh'$
defined by 
\[
\Upsilon_n(g_1\oplus h_1, \ldots, g_n\oplus h_n)\coloneq 
\Phi_n(g_1, \ldots, g_n)+
\Psi_n(h_1, \ldots, h_n)~. 
\]
This construction is actually the product in the category of curved $\Linfty$-algebras with their $\infty$-morphisms. 

\begin{proposition}[{\cite[Section~3]{DolgushevRogers17}}]\label{prop:SymMono}\leavevmode
\begin{enumerate}
\item The category $(\cLi, \oplus)$ is cartesian symmetric monoidal.
\item The Maurer--Cartan functor
is strong symmetric monoidal: 
\[\MC \colon (\mathsf{cont}\textsf{-} \mathsf{cL}_\infty, \oplus)
\to (\Set, \times)~. \]
\end{enumerate}
\end{proposition}

\begin{proof}
The proof relies on straightforward verifications. 
\end{proof}

\subsection{Enrichment over curved $\Linfty$-algebras}
The goal of this section is to enrich the category of $\Cobar \calC$-gebras 
over the symmetric monoidal category $(\cLi, \oplus)$ such that is gives back the category $\infty\textsf{-}\catofgebras{\Cobar\souche}$ under the Maurer--Cartan functor. 

\medskip

We consider the projection of the decomposition map of the coproperad $\calC$ onto the space of 2-level connected graphs with a total number of $n$ vertices: 
\[
	\Delta^{(n)} \colon \souche \xrightarrow{\Delta} \souche\boxtimes\souche \twoheadrightarrow (\souche\boxtimes\souche)^{(n)} \ .
\]

\begin{definition}[The enriched category $\cLi\textsf{-}\catofgebras{\Cobar\souche}$]
The structure of the category 
$\cLi\textsf{-}\catofgebras{\Cobar\souche}$ enriched over $(\cLi, \oplus)$ is defined as follows. 
\begin{description}
\item[\sc Objects] the $\Cobar \calC$-gebras $(A, \alpha)$.
\item[\sc Morphisms] the $\L_\infty$-algebras $\mathfrak{h}_{\alpha,\beta}\coloneq \left( \Hom_{\Sy}\big(\calC, \End^A_B\big), \{\mathcal{l}_n\}_{n\geqslant 1}\right)$ of \cref{cor:HoLieAlgeStructureBIS}. 
\item[\sc Compositions] the $\infty$-morphisms 
$\Phi^{\alpha, \beta, \gamma} \colon\h_{\beta,\gamma}\oplus \h_{\alpha, \beta} \rightsquigarrow \h_{\alpha, \gamma}$ given by $
\Phi^{\alpha, \beta, \gamma}_0\coloneq 0$, $\Phi^{\alpha, \beta, \gamma}_1\coloneq 0$, and by 
	\[ 
\Phi^{\alpha, \beta, \gamma}_n\left(
f_1\odot \cdots \odot f_n
\right)\coloneq		\souche \xrightarrow{\Delta^{(n)}} (\souche\boxtimes\souche)^{(n)} 
		\xrightarrow{f_1\odot \cdots \odot f_n}
				\big(\End^A_B\oplus\End^B_C\big)^{\boxtimes 2} 
		\twoheadrightarrow 
		\End^B_C\boxtimes\End^A_B 
		\to \End^A_C \ ,
	\]
	for $n\geqslant 2$\ , where the projection keeps only the composable diagrams and where the last arrow is the usual composite of functions.
\item[\sc Units] the $\infty$-morphisms $\Upsilon^{\alpha}\colon 0 \rightsquigarrow \h_{\alpha, \alpha}$ given by 
$\Upsilon^{\alpha}_0(1)\coloneq \id_A$ and 
by $\Upsilon^{\alpha}_n\coloneq 0$, for $n\geqslant 1$\ . 
\end{description}
\end{definition}

\begin{theorem}\label{thm:EnrichCat}
Let $\calC$ be a conilpotent dg  coproperad.
\begin{enumerate}
\item The above assignment defines a category $\cLi\textsf{-}\catofgebras{\Cobar\calC}$ enriched over $(\cLi, \oplus)$. 

\item The image of the enriched category $\cLi\textsf{-}\catofgebras{\Cobar\calC}$ under the Maurer--Cartan functor is isomorphic to the category $\infty\textsf{-}\catofgebras{\Cobar\souche}$.
\end{enumerate}
\end{theorem}

We split the proof of this theorem into the following lemmata. 

\begin{lemma}\label{lem:L1}
The assignment  $\Phi^{\alpha, \beta, \gamma}$ is an $\infty$-morphism of $\Linfty$-algebras.
\end{lemma}

\begin{proof}
The left-hand side of \cref{eq:CurvedLinfty}, defining the notion of an $\infty$-morphism, evaluated on elements $g_1+f_1, \ldots, g_n+f_n$ of 
$\h_{\beta,\gamma}\oplus \h_{\alpha, \beta}$ amounts to applying first 
$\Delta^{(p+1)}$ on $\calC$ to produce elements of 
$(\calC\boxtimes\calC)^{(p+1)}$, then applying 
$\Cop{(q)}{}$,\,   $-\Cop{}{(q)}$ and, $-d_\calC$ for $q=1$ to any element, and finally applying and composing the elements $f_1, \ldots, f_n$, $g_1, \ldots, g_n$, $\alpha$, $\beta$, $\gamma$, and $\partial_\End$
accordingly. This produces the following five types of terms, described from top to bottom:
\begin{itemize}
\item[$\diamond$] the composition of one $\alpha$, a level of terms among $f_1, \ldots, f_n$ and a level of terms among $g_1, \ldots, g_n$ coming from the evaluation of elements of $\calC \boxtimes (\calC \boxtimes \calC)$ made up of $n+1$ vertices and only $1$ element of $\oC$ on the top level, 
\item[$\diamond$] the composition of a level of terms among $f_1, \ldots, f_n$, 
one $\beta$,  and a level of terms among $g_1, \ldots, g_n$ coming from the evaluation of elements of 
$\calC \boxtimes (\calC \boxtimes \calC)$ made up of $n+1$ vertices and only $1$ element of $\oC$ on the middle level, 
\item[$\diamond$] the composition of a level of terms among $f_1, \ldots, f_n$, 
one $\beta$,  and a level of terms among $g_1, \ldots, g_n$ coming from the evaluation of elements of 
$(\calC \boxtimes \calC) \boxtimes \calC$ made up of $n+1$ vertices and only $1$ element of $\oC$ on the middle level, 
\item[$\diamond$] the composition of a level of terms among $f_1, \ldots, f_n$, 
a level of terms among $g_1, \ldots, g_n$, and one $\gamma$, coming from the evaluation of elements of 
$(\calC \boxtimes \calC) \boxtimes \calC$ made up of $n+1$ vertices and only $1$ element of $\oC$ on the bottom level, 
\item[$\diamond$] the term $-\sum_{i=1}^n\Phi^{\alpha, \beta, \gamma}_n\big(
(g_1+f_1)\odot \cdots \odot 
\partial(g_i+f_i)
\odot \cdots \odot (g_n+f_n)\big)$~. 
\end{itemize}
The first and third terms appear with a plus sign, since they come from the application of operations $\DD_{q+1}$ and the second and fourth terms appear with a minus sign, since they come from the application of operations $\UU_{q+1}$. Therefore, the coassociativity of the coproduct of the coproperad $\calC$ implies that the second terms cancel with the third terms. 

\medskip

Similarly the right-hand side of \cref{eq:CurvedLinfty} evaluated on elements $g_1+f_1, \ldots, g_n+f_n$ of 
$\h_{\beta,\gamma}\oplus \h_{\alpha, \beta}$ amounts to applying first 
$\Cop{(k)}{}$,\,   $-\Cop{}{(k)}$, and $-d_\calC$ for $k=1$, then applying $\Delta^{(i_1)}, \ldots, \Delta^{(i_k)}$ to the bottom terms produced by $\Cop{(k)}{}$ and to the top terms produced by $-\Cop{}{(k)}$
and finally applying and composing the elements $f_1, \ldots, f_n$, $g_1, \ldots, g_n$, $\alpha$, $\gamma$, and $\partial_\End$
accordingly. This produces the following three types of terms:
\begin{itemize}
\item[$\diamond$] the composition of one $\alpha$, a level of terms among $f_1, \ldots, f_n$ and a level of terms among $g_1, \ldots, g_n$ coming from the evaluation of elements of $(\calC \boxtimes \calC) \boxtimes \calC$ made up of $n+1$ vertices and only $1$ elements of $\oC$ on the top level, 
\item[$\diamond$] the composition of a level of terms among $f_1, \ldots, f_n$, 
a level of terms among $g_1, \ldots, g_n$, and one $\gamma$, coming from the evaluation of elements of 
$\calC \boxtimes (\calC \boxtimes \calC)$ made up of $n+1$ vertices and only $1$ elements of $\oC$ on the bottom level, 
\item[$\diamond$] the term 
$-\partial\left(\Phi^{\alpha, \beta, \gamma}_n\big(
(g_1+f_1)\odot \cdots \odot 
(g_n+f_n)\big)\right)$~. 
\end{itemize}
The two types of first terms come with a plus sign, since they are both produced by applying operations $\DD$, and agree by the coassociativity of the coproduct of the coproperad $\calC$, so do the terms of third type above and the terms of second type below, both coming with a minus sign, they are both produced by applying operations $\UU$.
%For instance, the two types of first terms are equal to the composite: 
%\[ 
%\calC \to (\calC\boxtimes\calC\boxtimes \calC)^{(n+1)} 
%\epi \big(\calC\boxtimes\calC\boxtimes \big(\I \oplus \underbrace{\oC}_{(1)}\big)\big)^{(n+1)} 
%\to 
%\End^B_C\boxtimes\End^A_B \boxtimes \End^A_A
%		\to \End^A_C \ ,
%	\]
%where the first map is the iteration $(\Delta \boxtimes \id_\calC)\Delta\cong (\id_\calC \boxtimes \Delta)\Delta$ of the coproduct of the coproperad $\calC$, 
%where the second map is the projection onto the summands made up of only one element of $\oC$ on the top level, 
%where the third map applies $k$ elements of $g_1, \ldots, g_n$ to the bottom level, the $n-k$ elements of  $f_1, \ldots, f_n$ labelled by different elements to the middle level, $\alpha$ to the non-trivial element of the top level, and the elements of $\I$ become the identity of $A$, 
%and where the last map is the usual composition of applications. 
We claim that 
\[\sum_{i=1}^n\Phi^{\alpha, \beta, \gamma}_n\big(
(g_1+f_1)\odot \cdots \odot 
\partial(g_i+f_i)
\odot \cdots \odot (g_n+f_n)\big)=\partial\left(\Phi^{\alpha, \beta, \gamma}_n\big(
(g_1+f_1)\odot \cdots \odot 
(g_n+f_n)\big)\right)
~.\]
The differential $\partial$ of the mapping space from the coproperad $\calC$ to an endomorphism space $\End$ is equal to 
$\partial(f) = \partial_\End \circ f - (-1)^{|f|} f \circ d_\calC$. In the above displayed relation, 
the terms coming on both sides from $\partial_\End$ agrees since it is a derivation on $\End$ and 
the terms coming from $\d_\calC$ agrees since it is a coderivation of the coproperad $\calC$. 
This concludes the proof of \cref{eq:CurvedLinfty}.

% \medskip

% It remains to show that the maps $\Phi^{\alpha, \beta, \gamma}_n$ preserve the filtrations, that is 

\end{proof}

\begin{lemma}\label{lem:L2}
	The composition $\Phi$ is associative, i.e. the following 
	 diagram of curved $\infty$-morphisms is  commutative:
	\[
		\begin{tikzcd}
			(\frakh_{\gamma,\delta}\oplus \frakh_{\beta,\gamma})\oplus \frakh_{\alpha, \beta} 
			\ar[rr, "\cong"]
			\ar[d,squiggly,"\Phi^{\,\beta, \gamma, \delta}\,\oplus\, \id"']
			& &
			\frakh_{\gamma,\delta}\oplus (\frakh_{\beta,\gamma}\oplus \frakh_{\alpha, \beta})
			\ar[d,squiggly,"\id\,\oplus\,\Phi^{\alpha, \beta,\gamma}"]
			\\
			\frakh_{\beta, \delta} \oplus \frakh_{\alpha, \beta}
			\ar[rd,squiggly,"\Phi^{\alpha, \beta,\delta}"']
			&& 
			\frakh_{\gamma, \delta} \oplus \frakh_{\alpha, \gamma}
			\ar[ld,squiggly,"\Phi^{\alpha, \gamma, \delta}"]
			\\
			& 
			\frakh_{\alpha,\delta}
			&
		\end{tikzcd}	\ .
	\]
\end{lemma}

\begin{proof}
Under the canonical isomorphism, we have to prove that 
\[\Phi^{\alpha, \beta, \delta}\,\circledcirc \left(\Phi^{\, \beta, \gamma, \delta}\,\oplus\, \id_{\h_{\alpha,\beta}}\right)
=\Phi^{\alpha, \gamma, \delta}\,\circledcirc \left(\id_{\h_{\gamma, \delta}} \,\oplus\, \Phi^{\alpha, \beta, \gamma}\right)~.  \]
Let us first notice that the curved $\infty$-morphism $\Phi^{\, \beta, \gamma, \delta}\,\oplus\, \id_{\h_{\alpha,\beta}}$
is equal to 
\begin{align*}
&\left(\Phi^{\, \beta, \gamma, \delta}\,\oplus\, \id_{\h_{\alpha,\beta}}\right)_0(1)=0\ , \quad 
\left(\Phi^{\, \beta, \gamma, \delta}\,\oplus\, \id_{\h_{\alpha,\beta}}\right)_1(h_1+g_1+f_1)=f_1\ , \quad \text{and} \\
&\left(\Phi^{\, \beta, \gamma, \delta}\,\oplus\, \id_{\h_{\alpha,\beta}}\right)_n(h_1+g_1+f_1, \ldots, h_n+g_n+f_n)=
\Phi^{\, \beta, \gamma, \delta}_n(h_1+g_1, \ldots, h_n+g_n)
\ , \ \text{for} \ n\geqslant 2\ ,
\end{align*}
and for elements $h_1+g_1+f_1, \ldots, h_n+g_n+f_n$ in 
$\h_{\gamma, \delta}\oplus \h_{\beta,\gamma}\oplus \h_{\alpha, \beta}$~. 
The curved $\infty$-morphism $\id_{\h_{\gamma, \delta}} \,\oplus\, \Phi^{\alpha, \beta, \gamma}$ satisfies a similar formula. 
The evaluation of the left-hand side on elements $h_1+g_1+f_1, \ldots, h_n+g_n+f_n$ of 
$\h_{\gamma, \delta}\oplus \h_{\beta,\gamma}\oplus \h_{\alpha, \beta}$ gives the composite 
\begin{multline*}
\calC \xrightarrow{(\Delta \boxtimes \id_\calC)\Delta} \big((\calC\boxtimes\calC)\boxtimes \calC\big)^{(n)}
		\xrightarrow{(h_1+g_1+f_1)\odot \cdots \odot (h_n+g_n+f_n)}
				\big(\End^A_B\oplus\End^B_C\oplus\End^C_D\big)^{\boxtimes 3} 
\\		\twoheadrightarrow 
		\End^C_D\boxtimes \End^B_C\boxtimes\End^A_B 
		\to \End^A_D 
\end{multline*}
and the evaluation of the right-hand side of the same elements gives a similar composite except that the first map is now the other iteration $(\id_\calC \boxtimes \Delta)\Delta$
of the coproduct of the coproperad $\calC$. So the coassociativity of this coproduct concludes the proof. 

\end{proof}

One can directly see that the assignment $\Upsilon^{\alpha}$ defines a curved $\infty$-morphism of $\Linfty$-algebras.

\begin{lemma}\label{lem:L3}
	The composition $\Phi$ is unital with respect to $\Upsilon$, i.e. the following 
	 diagram of curved $\infty$-morphisms is  commutative:
	 	\[
		\begin{tikzcd}[column sep=large]
			0 \oplus \frakh_{\alpha, \beta}
			\ar[rd, "\cong"]
			\ar[d,squiggly,"\Upsilon^{\beta}\,\oplus\, \id"']
			& &
			 \frakh_{\alpha, \beta}\oplus 0\ar[ld, "\cong"']
			\ar[d,squiggly,"\id\,\oplus\,\Upsilon^{\alpha}"]
			\\
			\frakh_{\beta, \beta} \oplus \frakh_{\alpha, \beta}
			\ar[r,squiggly,"\Phi^{\alpha, \beta, \beta}"']
			& \h_{\alpha, \beta}& 
			\frakh_{\alpha, \beta} \oplus \frakh_{\alpha, \alpha}
			\ar[l,squiggly,"\Phi^{\alpha, \alpha, \beta}"]
		\end{tikzcd}	\ .
	\]
\end{lemma}

\begin{proof}
Let us prove in details the equality $\Phi^{\alpha, \alpha, \beta}\cc\left(
\id_{\h_{\alpha, \beta}}\oplus \Upsilon^{\alpha}
\right)=\id_{\h_{\alpha,\beta}}$, the other one, namely 
$\Phi^{\alpha, \beta, \beta}\cc\left(
\Upsilon^{\beta}\oplus \id_{\h_{\alpha, \beta}}
\right)=\id_{\h_{\alpha,\beta}}$, is proved with the same kind of arguments. 
The curved $\infty$-morphism $\id_{\h_{\alpha, \beta}}\oplus \Upsilon^{\alpha}$ is equal to 
\begin{align*}
&\left(\id_{\h_{\alpha, \beta}}\oplus \Upsilon^{\alpha}\right)_0(1)=\id_A\ , \quad 
\left(\id_{\h_{\alpha, \beta}}\oplus \Upsilon^{\alpha}\right)_1(f)=f\ , \ \text{for} \ f\in \h_{\alpha, \beta}\ , \quad \text{and} \\
&\left(\id_{\h_{\alpha, \beta}}\oplus \Upsilon^{\alpha}\right)_n=0\ , \ \text{for} \ n\geqslant 2\ .
\end{align*}

Since $\Phi^{\alpha, \alpha, \beta}_k\big(\id_A^{\odot k}\big)=0$, for any $k\geqslant 0$, 
the component of arity $n=0$ of the left-hand side vanishes. 
Let us show that the component of arity $n=1$ is equal to the identity. 
For any $f\in \h_{\alpha, \beta}$, we have 
\[\left(\Phi^{\alpha, \alpha, \beta}\cc\left(
\id_{\h_{\alpha, \beta}}\oplus \Upsilon^{\alpha}
\right)\right)_1(f)=
\sum_{k\geqslant 0}\Phi^{\alpha, \alpha, \beta}_{k+1} \big(\id_A^{\odot k}\odot f\big)\ .
\]
The image of an element of $\calC$ under $\Phi^{\alpha, \alpha, \beta}_{k+1} \big(\id_A^{\odot k-1}\odot f\big)$ 
amounts to applying $\Delta^{(k+1)}$, which produces 2-level connected graphs with $k+1$ vertices labeled by $\calC$, keeping only $k$ elements of $\I\subset \calC$ at the top, that we replace by $\id_A$, and one element of $\calC(m,k)$ at the bottom, to which we apply $f(m,k)$, and finally composing the appearing functions. This shows that 
$\sum_{k\geqslant 0}\Phi^{\alpha, \alpha, \beta}_{k+1} \big(\id_A^{\odot k}\odot f\big)=f$\ .
In the end, we claim that the component of arity $n\geqslant 2$ vanishes. 
The component of arity $n$ of the composite of these two curved $\infty$-morphisms is equal to 
\[\left(\Phi^{\alpha, \alpha, \beta}\cc
\left(\id_{\h_{\alpha, \beta}}\oplus \Upsilon^{\alpha}\right)
\right)_n=
 \sum_{i_1+\cdots+i_k=n}
\Phi^{\alpha, \alpha, \beta}_k\circ \left(\left(\id_{\h_{\alpha, \beta}}\oplus \Upsilon^{\alpha}\right)_{i_1}, 
\ldots, \left(\id_{\h_{\alpha, \beta}}\oplus \Upsilon^{\alpha}\right)_{i_k}\right)
\ .\]
Since $\left(\id_{\h_{\alpha, \beta}}\oplus \Upsilon^{\alpha}\right)_i=0$, for $i\geqslant 2$, this composite is actually equal to 
\[\left(\Phi^{\alpha, \alpha, \beta}\cc
\left(\id_{\h_{\alpha, \beta}}\oplus \Upsilon^{\alpha}\right)
\right)_n(f_1\odot\cdots \odot f_n)=
\sum_{k\geqslant 0}\Phi^{\alpha, \alpha, \beta}_{k+n} \big(\id_A^{\odot k}\odot f_1\odot\cdots \odot f_n\big)\ .
\]
The image of an element of $\calC$ under $\Phi^{\alpha, \alpha, \beta}_{k+n} \big(\id_A^{\odot k}\odot f_1\odot\cdots \odot f_n\big)$ 
amounts to applying $\Delta^{(k+n)}$ which produces 2-level connected graphs with $k+n$ vertices labeled by $\calC$, keeping $k$ elements of $\I\subset \calC$ at the top, that we replace by $\id_A$, and $n$ elements of $\calC$ at the bottom, to which we apply $f_1\odot\cdots \odot f_n$, and finally composing the appearing functions. But for $n\geqslant 2$, this is impossible since this would force the underlying graphs to be disconnected. 
\end{proof}

\begin{remark}\label{rem:NotCont}
Before proving \cref{thm:EnrichCat}, let us emphasize the following crucial point: the composition $\infty$-morphisms are \emph{not}  continuous since 
$\Phi^{\alpha,\alpha, \alpha}_2(\id_A, \id_A)$ lives in $\scrF_1 \h_{\alpha, \alpha}$ but not in $\scrF_2 \h_{\alpha, \alpha}$~, for $\id_A\in 
\scrF_1 \h_{\alpha, \alpha}$~.
Moreover, there is no way to fix this: even if one wishes to consider another filtration of the $\Linfty$-algebras $\h_{\alpha,\beta}$~, the identities $\id_A$ have to sit in $\scrF_1  \h_{\alpha, \alpha}$
since 
$\id_A=\Upsilon^{\alpha}_0(1)\in \scrF_1  \h_{\alpha, \alpha}$~. 
This failure is not dramatic as we will see  in the proof below. 
\end{remark}

\begin{proof}[Proof of \cref{thm:EnrichCat}]\leavevmode
\begin{enumerate}
\item Lemmata~\ref{lem:L1} (well-defined), \ref{lem:L2} (associativity), and \ref{lem:L3} (unitality) 
prove the various axioms for the category $\cLi\textsf{-}\catofgebras{\Cobar\calC}$ to be enriched over $(\cLi, \oplus)$. 

\item \cref{prop:MCInfMor} shows that the spaces of morphisms, like $\h_{\alpha, \beta}$, of the former enriched category 
are sent to the sets of $\infty$-morphisms, from $(A, \alpha)$ to $(B, \beta)$. 
Since the composition $\infty$-morphisms $\Phi^{\alpha, \beta, \gamma}$ are not continuous, see \cref{rem:NotCont}, we cannot simply apply Point (2) of \cref{prop:SymMono}: we first have to show that 
$\MC\left(\Phi^{\alpha, \beta, \gamma}\right)$ is well-defined. 
For any $f\in \MC(\h_{\alpha, \beta})$ and any $g\in \MC(\h_{\beta, \gamma})$, 
the term ${\textstyle \frac{1}{n!}}
\Phi^{\alpha, \beta, \gamma}_n(g+f, \ldots, g+f)$ is equal to the summand of $g \circledcirc f$ 
corresponding to the application of $\Delta^{(n)}$. Since the image of any $c\in \calC$ under the decomposition map $\Delta$ is made up of a finite sum of terms, we have 
\[
\MC\left(\Phi^{\alpha, \beta, \gamma}\right)(g+f)(c)=\sum_{n\geqslant 2}^{N}{\textstyle \frac{1}{n!}}
\Phi^{\alpha, \beta, \gamma}_n(g+f, \ldots, g+f)(c)=(g \circledcirc f)(c)~,\]
for some $N\in \NN$. 
This shows that the image of  the $\infty$-morphism 
$\Phi^{\alpha, \beta, \gamma} \colon\h_{\beta,\gamma}\oplus \h_{\alpha, \beta} \rightsquigarrow \h_{\alpha, \gamma}$ of 
$\Linfty$-algebras under 
the Maurer--Cartan functor is well-defined and produces the composite $\circledcirc$ of $\infty$-morphisms of $\Omega \calC$-gebras. 
It remains to see that the image of the $\infty$-morphism 
$\Upsilon^{\alpha}\colon 0 \rightsquigarrow \h_{\alpha, \alpha}$ of 
$\Linfty$-algebras under 
the Maurer--Cartan functor is equal to the identity of $A$, that is
\[
\MC\left(\Upsilon^{\alpha}\right)(0)=\sum_{n\geqslant 0}{\textstyle \frac{1}{n!}}
\Upsilon^{\alpha}_n(0, \ldots, 0)=\Upsilon^{\alpha}_0(1)=\id_A~.\]
\end{enumerate}
\end{proof}

\section{Simplicial enrichment}
Following the method introduced by Dolgushev--Hoffnung--Rogers \cite{DolgushevHoffnungRogers14} on the operadic level, we apply the Deligne--Hinich integration functor of $\Linfty$-algebras to further enrich the category of $\Omega \calC$-gebras and $\infty$-morphisms with a simplicial structure. All the arguments of \emph{loc. cit.} apply \emph{mutatis mutandis} on the properadic level except at a certain key point where the fact that we are working with algebraic structures with possibly several outputs breaks the flow. 
In order to bypass this failure, we introduce other types of methods, based on the homotopy properties (model structures) of morphisms of 2-colored dg properads. In the end, this allows us to show that the localisations of the category of $\Omega \calC$-gebras 
at quasi-isomorphisms and at $\infty$-quasi-isomorphisms are both given by the homotopy category of its simplicial enrichment.

\subsection{Integration of curved $\Linfty$-algebras}\label{subsec:Integration}
The polynomial differential forms on the geometrical simplicies form a simplicial dg commutative algebra $\Omega_\bullet$  due to D. Sullivan in \cite{Sullivan77}. 
Recall that the (complete) tensor product with a dg commutative algebra produces an endofunctor of the category of curved $\Linfty$-algebras. 

\begin{definition}[Deligne--Hinich integration functor]
The \emph{Deligne--Hinich integration functor} is defined by the following composite 
\[
\begin{tikzcd}[column sep=large]
\MC_\bullet \ \colon \ (\mathsf{cont}\textsf{-}\catofalgebras{\cLi},\oplus) \ar[r, "- \widehat{\otimes}\, \Omega_\bullet"]& 
(\mathsf{cont}\textsf{-}\catofalgebras{\scLi},\oplus)  \ar[r, "\MC"]& 
(\sSet, \times) \ ,
\end{tikzcd}
\]
where $\mathsf{cont}\textsf{-}\catofalgebras{\scLi}$ stands for the category of simplicial curved $\Linfty$-algebras with their simplicial continuous $\infty$-morphisms, that is 
\[\MC_\bullet (\frakg)\coloneq \MC \left(\frakg\, \widehat{\otimes}\,  \Omega_\bullet\right)~.\]
\end{definition}

\begin{theorem}[\cite{Hinich97, Getzler09, DolgushevRogers17}]\leavevmode
\begin{enumerate}
\item The Deligne--Hinich integration functor is strong symmetric monoidal. 

\item The image of the Deligne--Hinich integration functor lies in the subcategory of Kan complexes. 
\end{enumerate}
\end{theorem}

\begin{proof}\leavevmode
\begin{enumerate}
\item It is straightforward to check that tensoring with a dg commutative algebra produces a strong symmetric functor. 
The strong symmetric monoidal functor $\MC$ (\cref{prop:SymMono}) induces a strong symmetric monoidal functor on the simplicial level since the monoidal structure on simplicial curved $\Linfty$-algebras is degree-wise. 

\item This is done in \cite[Proposition~2.2.3]{Hinich97} for dg Lie algebras, in 
\cite[Section~4]{Getzler09} for nilpotent $\Linfty$-algebras, and in \cite[Proposition~4.1]{DolgushevRogers17} for complete $\Linfty$-algebras. The arguments of the last two references apply as well to the curved case. 
\end{enumerate}
\end{proof}

\begin{remark}
The chosen  terminology comes from the fact that the Deligne--Hinich functor generalises the classical integration functor from Lie theory which associates a group to a (nilpotent or complete) Lie algebra: it assigns here a group "up to homotopy" (Kan complexes also known as $\infty$-groupoids) to a Lie algebra "up to homotopy" (curved $\Linfty$-algebra). 
\end{remark}

One properadic example has recently been studied in details by Kontsevich--Takeda--Vlassopulous in \cite{KTV23}. 
Recall that the notion of a pre-Calabi--Yau algebra \cite{TZ07Bis, KTV25} is encoded by the (properadic) cobar construction of a codioperad  $\calC_{\mathrm{pCY}}$, see also \cite[Section~3]{LV23}. 
The homotopy type of the restriction of the $\infty$-groupoid $\MC_\bullet (\frakg_{\calC_{\mathrm{pCY}},A})$ integrating the deformation dg Lie algebra encoding pre-Calabi--Yau algebras to non-degenerate ones is described in \cite[Section~4]{KTV23} in terms of the negative cyclic homology of the smooth algebra $A$. In this direction, it would interesting to study the $\infty$-groupoid 
$\MC_\bullet \left(\frakk_{\calC_{\mathrm{pCY}},A,B}\right)$
integrating pair of pre-Calabi--Yau algebras related by $\infty$-morphisms. 

\subsection{Simplicial category of homotopy gebras}Let $\calC$ be a conilpotent dg  coproperad.

\begin{definition}[The simplicial category $\Delta\textsf{-}\catofgebras{\Cobar\calC}$]
	The simplicial category $\Delta\textsf{-}\catofgebras{\Cobar\calC}$ is made up of the $\Cobar\calC$-gebras as objects and the simplicial sets 
	\[
		\Map(\alpha, \beta) \coloneq \MC_\bullet(\frakh_{\alpha,\beta}) 
	\]
as morphisms; the composition $U_{\alpha, \beta,\gamma}$ given by 
	\[
		U_{\alpha,\beta,\gamma}\coloneq \MC_\bullet(\Phi^{\alpha, \beta, \gamma}) \colon 
		\Map(\beta, \gamma)\times \Map(\alpha, \beta)
		\too
		\Map(\alpha, \gamma)
	\]
	and the unit given by $\MC_\bullet\left(\Upsilon^\alpha\right)$.
\end{definition}

The same argument as in the proof of  \cref{thm:EnrichCat} holds here: even if the composition $\infty$-morphism 
$\Phi^{\alpha, \beta, \gamma}$ fails to be continuous its image under the Maurer--Cartan functor makes sense and all the axioms of a simplicial category are straightforward to check. 

\begin{proposition}\label{prop:SimplicialCat0Simpl}
The category defined by the $0$-simplicies of the simplicial category $\Delta\textsf{-}\catofgebras{\Cobar\calC}$ is isomorphic 
to the category $\infty\textsf{-}\catofgebras{\Cobar\souche}$.
\end{proposition}

\begin{proof}
This is a straightforward corollary the fact that the dg commutative algebra $\Omega_0\cong \k$ is trivial and 
 Point (2) of  \cref{thm:EnrichCat}.
\end{proof}

\begin{remark}
There is another smaller but homotopy equivalent way to integrate $\L_\infty$-algebras into \emph{algebraic} $\infty$-groupoids using a functor $\mathrm{R}_\bullet$ which has also the advantage that the $1$-simplexes are the algebraic gauges, see \cite{Getzler09, RNV19, RocaiLucio24}. 
But this functor $\mathrm{R}_\bullet$ is so far functorial only with respect to a refinement of the notion of $\infty$-morphism.
In order to produce a simplicial category with this functor, one should refine the structure maps of the category enriched in $\L_\infty$-algebras above. 
\end{remark}

Our goal  is to prove \cref{thm:HoCatViaSimpl} which claims that the localisation 
of the category  $\infty\text{-}\catofgebras{\Cobar\calC}$ of $\Cobar\calC$-gebras with respect $\infty$-quasi-isomorphisms is equivalent to the homotopy category $\pi_0\left(\Delta\textsf{-}\catofgebras{\Cobar\calC}\right)$ of this aforementioned simplicial category. To this extend, we will need the following first homotopical properties. 

\medskip

For any $\infty$-morphism $f \colon (A, \alpha) \rightsquigarrow  (B, \beta)$ of $\Cobar\calC$-gebras, that is a Maurer--Cartan element in 
$\frakh_{\alpha,\beta}$, we consider the $\infty$-morphism $\Xi^f  \colon \frakh_{\beta, \gamma} \rightsquigarrow 
\frakh_{\beta, \gamma} \oplus \frakh_{\alpha, \beta}$
of curved $\Linfty$-algebras 
 defined by 
\[\left(\Xi^f\right)_0(1)\coloneq f \, , \quad \left(\Xi^f\right)_1(g)\coloneq g\ , \quad \text{and} \quad \left(\Xi^f\right)_n\coloneq 0\]
and the  $\infty$-morphism $\Xi_f  \colon \frakh_{\gamma, \alpha} \rightsquigarrow 
\frakh_{\alpha, \beta} \oplus \frakh_{\gamma, \alpha}$ of curved $\Linfty$-algebras  defined by 
\[\left(\Xi_f\right)_0(1)\coloneq f \, , \quad \left(\Xi_f\right)_1(g)\coloneq g\ , \quad \text{and} \quad \left(\Xi_f\right)_n\coloneq 0~.\]
%Checking that these assignments satisfy Equation~\eqref{eq:CurvedLinfty} is straightforward. 
One can directly check that they are $\infty$-morphisms or notice that they are both the categorical product of
the identity with the ``constant''  $\infty$-morphism defined by the Maurer--Cartan $f$ in the target $\frakh_{\alpha,\beta}$~.

\begin{lemma}\label{lem:pullbackpushoutinfty}
For any $\infty$-quasi-isomorphism $f \colon (A, \alpha) \stackrel{\sim}{\rightsquigarrow}  (B, \beta)$
of $\Cobar\calC$-gebras, the pullback morphism 
\[f^*\coloneq \Phi^{\alpha, \beta,\gamma} \circledcirc \Xi^f \ \colon \frakh_{\beta, \gamma} \rightsquigarrow \frakh_{\alpha, \gamma}\]
is a filtered $\infty$-quasi-isomorphism of   $\Linfty$-algebras and the pushout morphism 
\[f_*\coloneq \Phi^{\gamma, \alpha, \beta} \circledcirc \Xi_f \ \colon \frakh_{\gamma, \alpha} \rightsquigarrow \frakh_{\gamma, \beta}\]
is a filtered $\infty$-quasi-isomorphism of   $\Linfty$-algebras.
\end{lemma}

\begin{proof}
Since the arguments are similar in both cases, we make them explicit only for the pullback morphism $f^*$. 
First, we show that this composite actually defines an $\infty$-morphism of $\Linfty$-algebras, that is 
\[\left(f^*\right)_0(1)=\sum_{k\geqslant 2} \Phi_k^{\alpha, \beta, \gamma}
\left(\left(\Xi^f\right)_0(1)\,, \ldots, \left(\Xi^f\right)_0(1)\right)=
\sum_{k\geqslant 2} \Phi_k^{\alpha, \beta, \gamma}
\left(0+f, \ldots, 0+f\right)=0~,\]
by definition.

\medskip

More generally, the $n$-th component of the $\infty$-morphism $f^*$ is equal to 
\begin{align*}
\left(f^*\right)_n(g_1, \ldots, g_n) &= \left(\Phi^{\alpha, \beta, \gamma} \circledcirc \Xi^f\right)_n(g_1, \ldots, g_n)\\
&=
\sum_{k\geqslant 0} \Phi_{n+k}^{\alpha, \beta, \gamma}
\Big(\left(\Xi^f\right)_1(g_1), \ldots, \left(\Xi^f\right)_1(g_n), \underbrace{\left(\Xi^f\right)_0(1),\ldots, \left(\Xi^f\right)_0(1)}_{k\ \text{times}}\Big) \\
&=
\sum_{k\geqslant 0} \Phi_{n+k}^{\alpha, \beta, \gamma}
\Big(g_1+0, \ldots, g_n+0, \underbrace{0+f,\ldots, 0+f}_{k\ \text{times}}\Big)~. 
\end{align*}
The element $\Phi_{n+k}^{\alpha, \beta, \gamma}\Big(g_1+0, \ldots, g_n+0, 0+f,\ldots, 0+f\Big)$ of $\h_{\alpha, \beta}$
amounts to considering the summand of the decomposition map $\Delta$ of $\calC$ with $k$ vertices on the top level, to which one applies $f$, and $n$ vertices on the bottom level, to which one applies $g_1, \ldots, g_n$ in all the possible ways. 

\medskip

This description allows us to establish that the $\infty$-morphism $f^*$ is continuous, that is 
\[
\left(f^*\right)_n(g_1, \ldots, g_n)\in \scrF_{i_1+\cdots+i_n}\h_{\beta, \gamma}\quad , \qquad \forall g_1 \in \scrF_{i_1} \h_{\beta, \gamma},\ldots, g_n \in \scrF_{i_n} \h_{\beta, \gamma}~.
\]
For any $c\in \calC$, we denote by $\sum \gra\big(c_1, \dots, c_k; c'_1, \ldots, c'_n\big)$ the summand of its image $\Delta(c)$ under the decomposition map of the coproperad $\calC$ made up of 
directed connected graphs $\gra$ on 2 levels
with the $k$ vertices on the top level labelled by the elements $c_1, \dots, c_k$ and with the $n$ vertices of the bottom level labelled by  the elements $c'_1, \ldots, c'_n$~. 
When $c\in \scrD_{i_1+\cdots+i_n-1}\calC$, we claim that there exists at least one $1\leqslant j\leqslant n$ such that $c'_j\in \scrD_{i_j-1}\calC$~. 
Otherwise, one gets a contradiction by the exact same argument as the one given at the end of the proof of \cref{prop:CompHoLieAlgeStructure}.
This shows that $g_j\big(c'_j\big)=0$ and then that $\left(f^*\right)_n(g_1, \ldots, g_n)(c)=0$.

\medskip

It remains to show that the restriction of $\left(f^*\right)_1$ to 
$\scrF_k \frakh_{\beta, \gamma} \xrightarrow{\sim} \scrF_k \frakh_{\alpha,\gamma}$ is a quasi-isomorphism, for any $k\geqslant 1$. 
The above-mentioned description 
\[
\left(f^*\right)_1(g)= 
\sum_{k\geqslant 1} \Phi_{1+k}^{\alpha, \beta, \gamma}
(g, f, \ldots, f)
\]
shows that the element $\left(f^*\right)_1(g)\in \h_{\alpha, \gamma}$ amounts to first applying the decomposition map $\Delta$ of the coproperad $\calC$, then projecting onto the summand made up of exactly one vertex on the bottom part, applying $g$ and the bottom level and $f$ everywhere on the top level, and finally composing the upshot in $\End^A_C$.
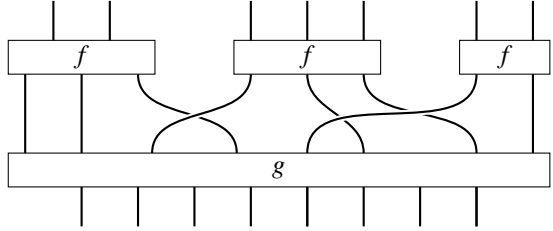
\[
\begin{tikzpicture}[scale=0.75]
	\draw[thick] (2,0.7) to[out=270,in=90] (3.75,-0.7) ;
	\draw[thick] (4,-1)-- (4,-2);
	\draw[thick] (3,-1)-- (3,-2);
	\draw[thick] (6,-1)-- (6,-2);
	\draw[thick] (1,1) to (1,-2);
	\draw[thick] (0.5,2) to  (0.5,1);
	\draw[thick] (0,1) to  (0,-1);
	\draw[thick] (2,-1) to  (2,-2);	
	\draw[thick] (7,-1) to  (7,-2);	
	\draw[thick] (8,-1) to  (8,-2);			
	\draw[thick] (6,1) to  (6,2);				
	\draw[thick] (6,0.7) to[out=270,in=90] (8,-0.7) to[out=270,in=90] (8,-2); 
	\draw[thick] (5,2) to  (5,1);
	\draw[thick] (5,-1) to (5,-2);
	\draw[draw=white,double=black,double distance=2*\pgflinewidth,thick] (5,0.7) to[out=270,in=90]  (6,-0.7);	
	\draw[draw=white,double=black,double distance=2*\pgflinewidth,thick] (9,2) to  (9,-1);
	\draw[thick] (1.5,2) to  (1.5,1);
	\draw[draw=white,double=black,double distance=2*\pgflinewidth,thick] (8,0.7) to[out=270,in=90] (5,-0.7); 
	\draw[thick] (5,-1) -- (5,-2);	
	\draw[draw=white,double=black,double distance=2*\pgflinewidth,thick] (4,2) to[out=270,in=90] (4,1);
	\draw[draw=white,double=black,double distance=2*\pgflinewidth,thick]  (4,0.7) to[out=270,in=90] (2.25,-0.7);
	\draw[thick] (8,2) to (8,1) ;
	\draw[fill=white] (-0.3,0.7) rectangle (2.3,1.3);
	\draw[fill=white] (3.7,0.7) rectangle (6.3,1.3); 
	\draw[fill=white] (7.7,0.7) rectangle (9.3,1.3); 
	\draw[fill=white] (-0.3,-1.3) rectangle (9.3,-0.7);
	\draw (1,1) node {\small{$f$}};
	\draw (5,1) node {\small{$f$}};
	\draw (8.5,1) node {\small{$f$}};
	\draw (4.5,-1) node {\small{$g$}};
\end{tikzpicture}\]
We extend the coradical filtration of the coaugmentation coideal $\oC$ to the whole coproperad $\calC$ by placing 
the counital element $\id \in I\subset \calC$ in $\scrR_0 \calC$. This produces again an increasing and exhaustive filtration 
\[
0=\scrR_{-1} \calC\subset \scrR_{0}\calC=I \subset 
\scrR_1 \calC\subset \scrR_2 \calC\subset \cdots \subset \scrR_l \calC\subset \scrR_{l+1} \calC\subset \cdots 
\quad \& \quad 
\bigcup_{l \geqslant 0} \scrR_l \calC =\calC
\ .
\]
We denote by $\mathrm{F}_l$ the induced decreasing and complete filtration on any mapping space with source $\calC$: 
$g\in \mathrm{F}_l$ if its restriction $g|_{\scrR_{l-1} \calC}= 0$ vanishes. For instance, we have 
\[
\scrF_k \frakh_{\beta, \gamma}=
\mathrm{F}_0 \scrF_k \frakh_{\beta, \gamma} \supset \mathrm{F}_1 \scrF_k \frakh_{\beta, \gamma}
 \supset \mathrm{F}_2 \scrF_k \frakh_{\beta, \gamma} 
 \supset \cdots \supset \mathrm{F}_l \scrF_k \frakh_{\beta, \gamma} \supset \mathrm{F}_{l+1}\scrF_k \frakh_{\beta, \gamma} \supset \cdots~,\]
where $\mathrm{F}_l \scrF_k \frakh_{\beta, \gamma} \coloneq \mathrm{F}_l \frakh_{\beta, \gamma} \cap \scrF_k \frakh_{\beta, \gamma}$~. 
Let us show that the differential of $\h_{\beta, \gamma}$ preserves this new filtration. The image of any any $g\in \mathrm{F}_l \h_{\beta, \gamma}$ under the (twisted) differential denoted by $\mathcal{l}_1$ in \cref{cor:HoLieAlgeStructureBIS} is made up of the following terms. The first term is equal to $\partial \circ g\in \mathrm{F}_l \h_{\beta, \gamma}$, where $\partial$ stands for the differential map of $\End^B_C$~. The second term  $g \circ d_\calC$ also lives in $\mathrm{F}_l \h_{\beta, \gamma}$ since 
\[d_\calC\left(\scrR_{l-1} \calC\right)\subset \scrR_{l-2}\calC\subset \scrR_{l-1} \calC~,\] by the conilpotency of the dg coproperad $\calC$ and by $d_\calC(\id)=0$. It remains to show that the other two terms 
$\DD_2(g,s\beta)$ and $\UU_2(s\gamma, g)$ also send $\scrR_{l-1} \calC$ to $0$. Both send the element $\id$ to $0$. 
We apply again the same argument as in the proof of \cref{prop:CompHoLieAlgeStructure}: the image of any element $c\in \scrR_{l-1}\calC$  under the decomposition map $\Cop{(1)}{}$ (respectively $\Cop{}{(1)}$) produces terms of the form $\gra(c_1, c_2)$ where $\gra$ is a directed connected graph with two vertices labelled on the bottom by $c_1$ and on the top by $c_2$, with $c_1\in \scrR_{l-1}\calC$ (respectively $c_2\in \scrR_{l-1}\calC$). 

\medskip

Applying the same kind of arguments once again, one can see that the map $\left(f^*\right)_1$ preserves the filtration $\mathrm{F}_l$, that is 
\[ \left(f^*\right)_1 \colon \mathrm{F}_l \scrF_k \frakh_{\beta, \gamma} \to  \mathrm{F}_l  \scrF_k \frakh_{\alpha,\gamma}~.\]
The differential on the first page $E^0$ of the associated spectral sequence is equal to the first above-mentioned term, that is 
$\partial \circ g$ and $\partial \circ \left(f^*\right)_1(g)$ respectively. The induced map $E^0\left(\left(f^*\right)_1\right)$ on this page is equal to the pullback by the first component $f_{(0)}$ of $f$: 
\[E^0\left(\left(f^*\right)_1\right)(g) \colon \calC \xrightarrow{g} \End^B_C \xrightarrow{\big(f_{(0)}\big)^*} \End^A_C~, \]
which is a quasi-isomorphism since $f_{(0)}$ is. 
Since the respective filtrations are complete and exhaustive, the Eilenberg--Moore comparison theorem 
\cite[Theorem~5.5.11]{WeibelBook} applies and shows that the 
the restriction of $\left(f^*\right)_1 \colon  
\scrF_k \frakh_{\beta, \gamma} \to \scrF_k \frakh_{\alpha, \gamma}$  is a quasi-isomorphism, for any $k\geqslant 1$. This concludes the proof that 
$f^*$ is a filtered $\infty$-quasi-isomorphism of $\Linfty$-algebras.
\end{proof}

The pullback map $f^*$ and the pushout map $f_*$ are two continuous $\infty$-morphisms of $\Linfty$-algebras defined in order to lift respectively the pullback and the pushout of $\infty$-morphisms of $\Omega \calC$-gebras by $f$: 
\[
\MC\big(f^*\big)(g)=g \circledcirc f \qquad \text{and} \qquad \MC\big(f_*\big)(g)=f \circledcirc g~. 
\]

\begin{proposition}[Homotopical invertibility of $\infty$-quasi-isomorphisms]\label{thm:HoInv}
For any $\infty$-quasi-isomorphism $f \colon (A, \alpha) \stackrel{\sim}{\rightsquigarrow}  (B, \beta)$ of $\Cobar\calC$-gebras, there exists an  $\infty$-morphism $g \colon (B, \beta) \rightsquigarrow (A, \alpha)$ such that the composite 
$g\cc f$ is homotopic to the $0$-simplex $\id_A$ of $\MC_\bullet(\h_{\alpha, \alpha})$ and such that 
the composite 
$f\cc g$ is homotopic to the $0$-simplex $\id_B$ of $\MC_\bullet(\h_{\beta, \beta})$. 
Any such $\infty$-morphism $g$ is unique up to homotopy in $\MC_\bullet(\h_{\beta, \alpha})$~.
\end{proposition}

\begin{proof}
We apply to $\MC_\bullet\left(f^*\right)$ the homotopy invariance of the integration function, that is the generalisation of Goldman--Milson theorem, due to Dolgushev--Rogers \cite{DolgushevRogers15} which states that the image of a filtered $\infty$-quasi-isomorphism (\cref{lem:pullbackpushoutinfty}) under the Maurer--Cartan  functor is a weak equivalence of Kan complexes. 
This way, we get a bijection between the respective connected components, that is 
\[\pi_0\left(\MC_\bullet\left(f^*\right)\right) \colon \pi_0\left(\MC_\bullet\left( \frakh_{\beta, \alpha}\right)\right) \xrightarrow{\cong} \pi_0\left(\MC_\bullet\left( \frakh_{\alpha,\alpha}\right)\right)~. \]
The representative of the identity morphism $\id_{A}$ on the right-hand side admits a unique pre-image on the left-hand side represented by a $g\in \MC\left( \frakh_{\beta, \alpha}\right)$. In other words, this provides us with an $\infty$-morphism of $\Cobar\calC$-gebras $g \colon (B, \beta) \rightsquigarrow (A, \alpha)$ such that the composite 
$g\cc f$ is homotopic to the $0$-simplex $\id_A$ by the above remark. 
Considering the image under the map $\pi_0\left(\MC_\bullet\left(f_*\right)\right)$, one can see that $f\cc g\cc f$ is homotopic to $f$ in $\MC_\bullet\left( \frakh_{\alpha, \beta}\right)$. Then the bijection 
\[\pi_0\left(\MC_\bullet\left(f^*\right)\right) \colon \pi_0\left(\MC_\bullet\left( \frakh_{\beta, \beta}\right)\right) \xrightarrow{\cong} \pi_0\left(\MC_\bullet\left( \frakh_{\alpha,\beta}\right)\right) \]
sends both classes represented by $f \cc g$ and $\id_B$ to this sole class, which implies that 
$f \cc g$ is homotopic to $\id_B$ in 
$\MC_\bullet\left( \frakh_{\beta, \alpha}\right)$ by injectivity. This concludes the proof. 
\end{proof}

The following statement allows one to better understand the homotopy relation between $\infty$-morphisms of $\Omega\calC$-gebras defined by the Kan complexes $\MC_\bullet(\h_{\alpha, \beta})$~. 

\begin{proposition}\label{prop:HoEquiv}
Let $f, g \colon (A, \alpha) {\rightsquigarrow}  (B, \beta)$ be two $\infty$-morphisms of $\Cobar\calC$-gebras which are homotopy equivalent in $\MC_\bullet\left(\h_{\alpha, \beta}\right)$. 
They fit into a commutative diagram 
\[
\begin{tikzcd}[column sep=3em, row sep=2em]
	(A,\alpha) \arrow[d, squiggly, "f"] & (A',\alpha') \arrow[d, squiggly, "k"] \arrow[l, "\sim"'] \arrow[r, "\sim"] & (A,\alpha) \arrow[d, squiggly, "g"]\\
	(B, \beta) & (B', \beta') \arrow[l, "\sim"'] \arrow[r, "\sim"]& (B, \beta)
\end{tikzcd}
\]
in the category $\infty\text{-}\catofgebras{\Cobar\souche}$, where the horizontal maps are strict quasi-isomorphisms which admit respectively common sections: 
\[
\begin{tikzcd}[column sep=3em, row sep=2em]
	(A,\alpha) \arrow[r, "\sim", "i_A"', bend left=50, hook] & (A',\alpha')  \arrow[l, "\sim"', "g_A", two heads]   \arrow[r, "\sim", "h_A"', two heads] & 
	\arrow[l, "\sim"', "i_A", bend right=50, hook'] (A,\alpha) 
\end{tikzcd}\quad \text{and} \quad 
\begin{tikzcd}[column sep=3em, row sep=2em]
	(B,\beta) \arrow[r, "\sim", "i_B"', bend left=50, hook] & (B',\beta')  \arrow[l, "\sim"', "g_B", two heads]   \arrow[r, "\sim", "h_B"', two heads] & 
	\arrow[l, "\sim"', "i_B", bend right=50, hook'] (B,\beta)
\end{tikzcd}~,
\]
that is $g_Ai_A=h_Ai_A=\id_A$ and $g_Bi_B=h_Bi_B=\id_B$~.
\end{proposition}

\begin{proof}
We first recall that the Kan complex $\MC_\bullet\left(\h_{\alpha, \beta}\right)$ is given by 
\[\MC_\bullet\left(\h_{\alpha, \beta}\right)=\MC\left(\h_{\alpha, \beta}\,\widehat{\otimes}\,\Omega_\bullet\right)= 
\MC\left(\Hom_{\Sy}\big(\calC, \End^A_B\big)\,\widehat{\otimes}\,\Omega_\bullet\right)~.\]
Then, in the model category of \cref{prop:ModelCatColoreddgProperads}, 
the 2-colored dg properad $\left(\Cobar \calC\right)_{\bullet \rightsquigarrow \bullet}$ is cofibrant (\cref{prop:Res}) and the classical construction $\End_{A, B}\otimes \Omega_1$ provides us with a path object for $\End_{A, B}$~.
So right homotopies between morphisms from  
$\left(\Cobar \calC\right)_{\bullet \rightsquigarrow \bullet}$ to $\End_{A, B}$
correspond to 1-simplices in the mapping space
$\Hom_{\mathrm{2-colored\ dg\   properads}}\left(\left(\Cobar \calC\right)_{\bullet \rightsquigarrow \bullet}, \End_{A,B}\otimes \Omega_\bullet\right)$~.
The universal construction of a (simplicial) $\L_\infty$-algebra given in Point~(2) of \cref{thm:ShiftedLinfini}  applies as well to the (simplicial) 
2-colored dg properad $\End_{A, B}\otimes \Omega_\bullet$ and the canonical isomorphism 
\[
\Hom_{\Sy}\big(\calC, \End^A_B\big)\,\widehat{\otimes}\,\Omega_\bullet \cong 
\Hom_{\Sy}\big(\calC, \End^A_B\otimes \Omega_\bullet\big)\]
preserves the respective simplicial $\L_\infty$-algebra structures. 
This yields the following isomorphism of Kan complexes 
\[\MC_\bullet\left(\h_{\alpha, \beta}\right)=
\MC\left(\Hom_{\Sy}\big(\calC, \End^A_B\big)\,\widehat{\otimes}\,\Omega_\bullet\right)
\cong ~
\MC\left(\Hom_{\Sy}\big(\calC, \End^A_B\otimes \Omega_\bullet\big)\right).\]
In the end, this shows that  the two Maurer--Cartan elements corresponding to two $\infty$-morphisms $f, g \colon (A, \alpha) {\rightsquigarrow}  (B, \beta)$ of $\Cobar\calC$-gebras are homotopic in the Kan complex $\MC_\bullet\left(\h_{\alpha, \beta}\right)$ if and only if the two induced morphisms of 2-colored dg properads 
\[(\alpha, f, \beta)\,, (\alpha, g, \beta)\, \colon \left(\Cobar \calC\right)_{\bullet \rightsquigarrow \bullet} \to \End_{A, B}\]
are right homotopic in the model category of 2-colored dg properads.

\medskip

The classical model categorical properties imply that these two morphisms of 2-colored dg properads 
$(\alpha, f, \beta)$ and $(\alpha, g, \beta)$ are also left homotopic. 
Then we claim that Theorem~8.4 of \cite{Fresse10ter} also holds for 2-colored dg properads: one applies 
\emph{mutatis mutandis} the same arguments 
based on the notion of endomorphism dg props of diagrams  \cite[Section~6.3]{Fresse10ter}
in the 2-colored dg properads setting. 
One can literally use the arguments of \emph{loc.cit.}, especially \cite[Lemmata~8.2 and 8.3]{Fresse10ter}, by noticing that the underlying total chain complexes of the 2-colored properads of endomorphisms of diagrams made up of maps like $j,k \colon (A,B) \to \left(A', B'\right)$ are isomorphic to the underlying chain complexes of the properads of endomorphisms of similar diagrams under the 
assignment $j\oplus k \colon A \oplus B \to A' \oplus B'$.

As in the proof of \cref{prop:InftyIsot}, we consider the classical and functorial path objects in the category of chain complexes 

\[\begin{array}{l}
\begin{tikzcd}
	A \ar[r,hook, "\sim", "i_A"']  & 
	\mathrm{Path}(A)\cong(A\oplus A \oplus s^{-1}A, d)  \ar[r,two heads, "{(g_A,h_A)}"'] 
	& A \oplus A 
\end{tikzcd} 
\quad \text{and}\\
\begin{tikzcd}
	B \ar[r,hook, "\sim", "i_B"']  & 
	\mathrm{Path}(B)\cong(B\oplus B \oplus s^{-1}B, d)  \ar[r,two heads, "{(g_B,h_B)}"'] 
	& A \oplus A ~,
\end{tikzcd} 	
\end{array}\]
where $i_A(a)\coloneq (a,a,0)$ and $i_B(b)\coloneq (b,b,0)$ are quasi-isomorphisms. 
Therefore, $g_A$ and $h_A$ are quasi-isomorphisms sharing a common section $i_A$, 
that is $g_Ai_A=h_Ai_A=\id_A$, and similarly for $B$. 
This construction is monoidal, $\mathrm{Path}(A\oplus B)\cong \mathrm{Path}(A) \oplus \mathrm{Path}(B)$~,
and thus fits into the aforementioned remark. 

\medskip

The 2-colored dg properad $\left(\Cobar \calC\right)_{\bullet \rightsquigarrow \bullet}$ is cofibrant by \cref{prop:Res} and the pair $(A,B)$ of chain complexes is fibrant and cofibrant since we work over a field. 
The 2-colored properadic version of \cite[Theorem~8.4]{Fresse10ter} applies to the two left-homotopic morphisms 
$(\alpha, f, \beta)$ and $(\alpha, g, \beta)$ and to the zigzag
\[
\begin{tikzcd}[column sep=3em, row sep=2em]
	(A,B) 
	& 
	\mathrm{Path}((A, B)) 
	\arrow[l, two heads, "\sim"'] 
	\arrow[r, two heads, "\sim"] 
	& 
	(A,B) 
\end{tikzcd}
\]
in the product model structure
on $\dgVect \times \dgVect$ to produce a zigzag 
\[
\begin{tikzcd}[column sep=3em, row sep=2em]
	((A,B),  (\alpha, f, \beta)) 
	& 
	\mathrm{Path}((A, B), \phi) 
	\arrow[l, two heads, "\sim"'] 
	\arrow[r, two heads, "\sim"] 
	& 
	((A,B), (\alpha, g, \beta)) 
\end{tikzcd}
\]
in the category of $\left(\Cobar \calC\right)_{\bullet \rightsquigarrow \bullet}$-gebras. 
Unfolding these data, 
 it provides us first with
  two $\Omega\calC$-gebra structures $\alpha'$ on $\mathrm{Path}(A)$
and $\beta'$  on  $\mathrm{Path}(B)$ respectively, for which the underlying maps of the path objects are strict quasi-isomorphisms
of $\Omega\calC$-gebras:
\[
\begin{tikzcd}[column sep=3em, row sep=2em]
	(A,\alpha) \arrow[r, "\sim", "i_A"', bend left=50, hook] & (\mathrm{Path}(A),\alpha')  \arrow[l, "\sim"', "g_A", two heads]   \arrow[r, "\sim", "h_A"', two heads] & 
	\arrow[l, "\sim"', "i_A", bend right=50, hook'] (A,\alpha) 
\end{tikzcd}\ \text{and} \ 
\begin{tikzcd}[column sep=3em, row sep=2em]
	(B,\beta) \arrow[r, "\sim", "i_B"', bend left=50, hook] & (\mathrm{Path}(B),\beta')  \arrow[l, "\sim"', "g_B", two heads]   \arrow[r, "\sim", "h_B"', two heads] & 
	\arrow[l, "\sim"', "i_B", bend right=50, hook'] (B,\beta)
\end{tikzcd}~.
\]
(This is the upshot of the existence of the map $k \colon P=\left(\Cobar \calC\right)_{\bullet \rightsquigarrow \bullet} \to 
\mathrm{End}_{\mathbb{Y}}$
in \emph{loc. cit.}~.)
Then it provides us with 
an $\infty$-morphism 
$k \colon (\mathrm{Path}(A), \alpha') \rightsquigarrow (\mathrm{Path}(B), \beta')$ between them satisfying the following commutative diagram 
\[
\begin{tikzcd}[column sep=3em, row sep=2em]
	(A,\alpha) \arrow[d, squiggly, "f"] & (\mathrm{Path}(A),\alpha') \arrow[d, squiggly, "k"] \arrow[l, "\sim"', "g_A"] \arrow[r, "\sim", "h_A"'] & (A,\alpha) \arrow[d, squiggly, "g"]\\
	(B, \beta) & (\mathrm{Path}(B), \beta') \arrow[l, "\sim"', "g_B"] \arrow[r, "\sim", "h_B"']& (B, \beta)
\end{tikzcd}\]
in the category $\infty\text{-}\catofgebras{\Cobar\calC}$~.
\end{proof}

\begin{theorem}\label{thm:HoCatViaSimpl}
The canonical functor 
\[\rmH\, \colon \infty\text{-}\catofgebras{\Cobar\souche} \to \pi_0\left(\Delta\textsf{-}\catofgebras{\Cobar\calC}\right)\]
is the universal functor which sends quasi-isomorphisms (respectively $\infty$-quasi-isomorphisms) to isomorphisms: for any functor $\rmF\, \colon \allowbreak \infty\textsf{-}\catofgebras{\Cobar\souche} \to \mathsf{C}$ sending quasi-isomorphisms (respectively $\infty$-quasi-isomorphisms) to isomorphisms, there exists a unique functor 
$\rmG\, \colon \pi_0\left(\Delta\textsf{-}\catofgebras{\Cobar\calC}\right) \to \mathsf{C}$ such that  $\rmF=\rmG\rmH$. 
\end{theorem}

\begin{proof}
\cref{thm:HoInv} proves that the canonical functor 
\[ \infty\text{-}\catofgebras{\Cobar\souche} \to \pi_0\left(\Delta\textsf{-}\catofgebras{\Cobar\calC}\right)\]
sends $\infty$-quasi-isomorphisms to isomorphisms.
Let us now prove the universal property. Let  $\rmF\, \colon \allowbreak \infty\textsf{-}\catofgebras{\Cobar\souche} \to \mathsf{C}$ be a functor sending quasi-isomorphisms (respectively $\infty$-quasi-isomorphisms) to isomorphisms. 
The existence of a functor 
$\rmG\, \colon \pi_0\left(\Delta\textsf{-}\catofgebras{\Cobar\calC}\right) \to \mathsf{C}$ satisfying  $\rmF=\rmG\rmH$ is equivalent to the fact that any pair of $\infty$-morphisms $f, g \colon (A, \alpha) {\rightsquigarrow}  (B, \beta)$ of $\Cobar\calC$-gebras homotopy equivalent in 
$\MC_\bullet\left(\h_{\alpha, \beta}\right)$ share the same image $\rmF(f)=\rmF(g)$ under the functor $\rmF$. 
In order to settle this property, we apply \cref{prop:HoEquiv} to get commutative diagrams 
\[
\begin{tikzcd}[column sep=3em, row sep=2em]
	(A,\alpha) \arrow[d, squiggly, "f"] & (\mathrm{Path}(A),\alpha') \arrow[d, squiggly, "k"] \arrow[l, "\sim"', "g_A"] \arrow[r, "\sim", "h_A"'] & (A,\alpha) \arrow[d, squiggly, "g"]\\
	(B, \beta) & (\mathrm{Path}(B), \beta') \arrow[l, "\sim"', "g_B"] \arrow[r, "\sim", "h_B"']& (B, \beta) 
\end{tikzcd}\]
\[
\begin{tikzcd}[column sep=3em, row sep=2em]
	(A,\alpha) \arrow[r, "\sim", "i_A"', bend left=50, hook] & (A',\alpha')  \arrow[l, "\sim"', "g_A", two heads]   \arrow[r, "\sim", "h_A"', two heads] & 
	\arrow[l, "\sim"', "i_A", bend right=50, hook'] (A,\alpha) 
\end{tikzcd}\quad \text{and} \quad 
\begin{tikzcd}[column sep=3em, row sep=2em]
	(B,\beta) \arrow[r, "\sim", "i_B"', bend left=50, hook] & (B',\beta')  \arrow[l, "\sim"', "g_B", two heads]   \arrow[r, "\sim", "h_B"', two heads] & 
	\arrow[l, "\sim"', "i_B", bend right=50, hook'] (B,\beta)
\end{tikzcd}
\]
in the category $\infty\text{-}\catofgebras{\Cobar\souche}$.
Applying the functor $\rmF$ to the last two ones, we get 
\[\rmF(g_A)\rmF(i_A)=\rmF(h_A)\rmF(i_A)=\id_{\rmF(A)}\quad  \text{and} \quad  \rmF(g_B) \rmF(i_B)=\rmF(h_B)\rmF(i_B)=\id_{\rmF(B)}~.\]
Since $\rmF$ sends ($\infty$-)quasi-isomorphisms to isomorphisms, we obtain $\rmF(g_A)=\rmF(h_A)$ and $\rmF(g_B)=\rmF(h_B)$~.
Finally, the image of the first above displayed commutative diagram gives 
\[
\rmF(f)=\rmF(g_B)\rmF(k)\rmF(g_A)^{-1}=\rmF(h_B)\rmF(k)\rmF(h_A)^{-1}=\rmF(g)~,
\]
which concludes the proofs. 
\end{proof}

\begin{remark}
While the proof of \cref{lem:pullbackpushoutinfty} is a properadic generalisation of the operadic proof of 
\cite[Proposition~4.3]{DolgushevHoffnungRogers14} and while the proof of 
\cref{thm:HoInv} is borrowed from \cite[Corollary~4.4]{DolgushevHoffnungRogers14}, 
 this is \emph{not} the case for the proof of \cref{prop:HoEquiv} and thus of \cref{thm:HoCatViaSimpl}: the property that  
a path object of a mapping space $\End^A_B$ is a mapping space to a path object 
\[\mathrm{Path}\left(\End^A_B\right)\cong \End^A_{\mathrm{Path}(B)}~,\] which is true on the operadic level since one has only one output, does not hold on the properadic level. This explains why we had to introduce other methods based on homotopical properties (model structures) of 2-colored dg properads to bypass this failure. In the end, the present arguments settle a stronger version than the one of \emph{loc. cit.} that deals with the localisation with respect to quasi-isomorphisms. 
\end{remark}

\subsection{Simplicial homotopy of $\infty$-quasi-isomorphisms}
After \cref{thm:MainInftyQi}, we give  another homotopy characterisation of $\infty$-quasi-isomorphisms  in the framework 
of the simplicial category of $\Cobar \calC$-gebras this time. 

\begin{theorem}[Homotopy characterisation of $\infty$-quasi-isomorphisms]\label{thm:CaraInftyQI}
Let  $f \colon (A, \alpha) \rightsquigarrow (B, \beta)$ be an $\infty$-morphism of $\Omega\calC$-gebras. The following assertions are equivalent.
\begin{enumerate}
\item The $\infty$-morphism $f$ is an $\infty$-quasi-isomorphism. 

\item The pullback maps 
\[\MC_\bullet\left(f^*\right) \colon \MC_\bullet\left( \frakh_{\beta, \alpha}\right) \xrightarrow{\sim} \MC_\bullet\left( \frakh_{\alpha,\alpha}\right) \quad \text{and} \quad 
\MC_\bullet\left(f^*\right) \colon \MC_\bullet\left( \frakh_{\beta, \beta}\right) \xrightarrow{\sim} \MC_\bullet\left( \frakh_{\alpha, \beta}\right) \]
 are  weak equivalences of Kan complexes. 

\item The pushout maps 
\[
\MC_\bullet\left(f_*\right) \colon \MC_\bullet\left( \frakh_{\beta, \alpha}\right) \xrightarrow{\sim} \MC_\bullet\left( \frakh_{\beta, \beta}\right) 
\quad \text{and} \quad 
\MC_\bullet\left(f_*\right) \colon \MC_\bullet\left( \frakh_{\alpha, \alpha}\right) \xrightarrow{\sim} \MC_\bullet\left( \frakh_{\alpha,\beta}\right) 
\]
 are  weak equivalences of Kan complexes. 

\item All the pullback and pushout maps 
\[\MC_\bullet\left(f^*\right) \colon \MC_\bullet\left( \frakh_{\beta, \gamma}\right) \xrightarrow{\sim} \MC_\bullet\left( \frakh_{\alpha, \gamma}\right) \quad \text{and} \quad 
\MC_\bullet\left(f_*\right) \colon \MC_\bullet\left( \frakh_{\gamma, \alpha}\right) \xrightarrow{\sim} \MC_\bullet\left( \frakh_{\gamma,\beta} \right) \]
 are  weak equivalences of Kan complexes, for any $\Omega\calC$-gebra structure $(C, \gamma)$.  
\end{enumerate}
\end{theorem}

\begin{proof}
We have already seen that $(1)\Rightarrow (2)$, $(1)\Rightarrow (3)$, and $(1)\Rightarrow (4)$ are direct corollaries of \cref{lem:pullbackpushoutinfty} under the homotopy invariance of the integration functor \cite{DolgushevRogers15}. 
The implications $(4)\Rightarrow (2)$ and $(4)\Rightarrow (3)$ are obvious. It remains to prove $(2)\Rightarrow (1)$ and $(3)\Rightarrow (1)$. We now show the first one; the second one can be established by similar arguments. 

\medskip

Since the pullback map $\MC_\bullet\left(f^*\right) \colon \MC_\bullet\left(\frakh_{\beta, \alpha}\right) \xrightarrow{\sim} \MC_\bullet\left( \frakh_{\alpha,\alpha}\right)$ is a weak equivalence of Kan complexes, it induces a bijection 
$\pi_0\left(\MC_\bullet\left( \frakh_{\beta, \alpha}\right)\right) \cong  \pi_0\left(\MC_\bullet\left( \frakh_{\alpha,\alpha}\right)\right)$~. 
Therefore there exists an $\infty$-morphism $g\, \colon \beta \rightsquigarrow \alpha$ such that the composite $g \cc f$ is homotopy equivalent to the identity $\id_A$ in $\MC_\bullet\left(\frakh_{\alpha, \alpha}\right)$. 
\cref{prop:HoEquiv} shows that 
 the composite $g \cc f$ fits into a commutative diagram of the following form
\[
\begin{tikzcd}[column sep=3em, row sep=2em]
	(A,\alpha) \arrow[d, squiggly, "g\cc f"] & (A',\alpha') \arrow[d, squiggly, "k"] \arrow[l, "\sim"'] \arrow[r, "\sim"] & (A,\alpha) 
	\arrow[d, squiggly, "\id_A","\sim"']\\
	(A, \alpha) & (A', \alpha'') \arrow[l, "\sim"'] \arrow[r, "\sim"]& (A, \alpha)~.
\end{tikzcd}
\]
This shows that $g\cc f$ is an $\infty$-quasi-isomorphism and that $f$ induces a monomorphism on homology.

\medskip

Since the pushout map $\MC_\bullet\left(f_*\right) \colon \MC_\bullet\left(\frakh_{\alpha, \alpha}\right) \to \MC_\bullet\left( \frakh_{\alpha,\beta}\right)$ is a morphism of Kan complexes, the two homotopic $\infty$-morphisms
 $g\cc f$ and $\id_A$ are sent to the two $\infty$-morphisms 
$f \cc g \cc f$ and $f$, which are homotopic in $\MC_\bullet\left(\frakh_{\alpha,\beta}\right)$. 
Finally, since the pullback map 
$\MC_\bullet\left(f^*\right) \colon \MC_\bullet\left(\frakh_{\beta, \beta}\right) \xrightarrow{\sim} \MC_\bullet\left(\frakh_{\alpha, \beta}\right)$
is a weak equivalence of Kan complexes, it induces a bijection 
$\pi_0\left(\MC_\bullet\left( \frakh_{\beta, \beta}\right)\right) \cong  \pi_0\left(\MC_\bullet\left(\frakh_{\alpha,\beta}\right)\right)$~. 
This implies that the two homotopic $\infty$-morphisms 
$f \cc g \cc f$ and $f$ on the right-hand side come from the two $\infty$-morphisms 
$f \cc g$ and $\id_B$ on the left-hand side, which are thus homotopic in $\MC_\bullet\left(\frakh_{\beta,\beta}\right)$~.
Using \cref{prop:HoEquiv} again, one shows that 
$f \cc g$ is an $\infty$-quasi-isomorphism and that $f$ induces an epimorphism on homology, which concludes the proof. 
\end{proof}

Finally, one can give a third version of the homotopy transfer theorem for 
$\Cobar \calC$-gebras in the setting of the simplicial category structure settled above. The data of a quasi-isomorphism $f \colon A \xrightarrow{\sim} B$ of chain complexes is equivalent to a Maurer--Cartan element $(0,f,0)$ in the $\Linfty$-algebra 
$\mathfrak{k}_{\calC, A,B}$, where $f \colon \calC \to \End^A_B$ is concentrated in $\I$ and vanishes on $\oC$, 
by \cref{prop:MCInfty}. By the usual twisting procedure \cite[Section~4.4]{DSV18}, one gets an $\Linfty$-algebra 
$\mathfrak{k}_{\calC, A,B}^{\, (0,f,0)}$ whose Maurer--Cartan elements are triples $(\alpha, g, \beta)$ such that 
$\alpha$ is an $\Omega \calC$-gebra on $A$, 
$\beta$ is an $\Omega \calC$-gebra on $B$, and 
$f+g$ is an $\infty$-morphism between them. We consider its $\Linfty$-subalgebra ${\bar{\mathfrak{k}}}^{\, f}_{\calC, A, B}$ defined by 
\[
\left( s\Hom_{\Sy}\left(\oC, \End_A\right)\oplus \Hom_{\Sy}\left(\oC, \End^A_B\right) \oplus s\Hom_{\Sy}\left(\oC, \End_B\right), -\partial^{(0,f,0)}, \left\{\mathcal{k}_n^{(0,f,0)}\right\}_{n\geqslant 2}, \left\{\scrF_k\right\}_{k\geqslant 1}\right)~.
\]

\begin{theorem}[Homotopy transfer theorem]\label{Thm:HTTnew}
Given a quasi-isomorphism $f \colon A \xrightarrow{\sim} B$ of chain complexes and an $\Omega \calC$-gebra structure $\beta$ on $B$, there is a Maurer--Cartan element $(\alpha, F, \beta') \in {\bar{\mathfrak{k}}}^{\, f}_{\calC, A, B}$ unique up to homotopy in $\MC_\bullet\left({\bar{\mathfrak{k}}}^{\, f}_{\calC, A, B}\right)$ such that 
$\beta'$ is homotopic to $\beta$ in $\MC_\bullet\left(\g_{\calC, B}\right)$~. 
\end{theorem}

\begin{proof}
The proof is the same as \cite[Theorem~5.1]{DolgushevHoffnungRogers14}. We recall it quickly for self-completeness of the present text. 
We claim first that the canonical projection 
\begin{equation}\label{Eq:AcyclicFib}\tag{$\diamond$}
\begin{tikzcd}
\Pi_B \ \colon \ {\bar{\mathfrak{k}}}^{\, f}_{\calC, A, B}  \ar[r, two heads, "\sim"]	& s\g_{\calC, B}
\end{tikzcd}
\end{equation}
is a filtered $\infty$-quasi-isomorphism of $\Linfty$-algebras. 
Then, using the homotopy invariance of the integration functor \cite{DolgushevRogers15}, we get a weak equivalence (acyclic fibration actually) of Kan complexes 
\[
\begin{tikzcd}
\MC_\bullet(\Pi_B) \ \colon \ \MC_\bullet\left({\bar{\mathfrak{k}}}^{\, f}_{\calC, A, B}\right)  \ar[r, two heads, "\sim"]	& \MC_\bullet\left(s\g_{\calC, B}\right)
\end{tikzcd}
\]
and thus a bijection 
\[\pi_0\left(\MC_\bullet\left({\bar{\mathfrak{k}}}^{\, f}_{\calC, A, B}\right)\right) \cong \pi_0\left(\MC_\bullet\left(s\g_{\calC, B}\right)\right)~.\]
Finally, any element $[\beta]$ on the right-hand side admits a unique pre-image $[(\alpha, F, \beta')]$ on the left-hand side, which concludes the proof. 

\medskip

The proof of \eqref{Eq:AcyclicFib} is similar than \emph{loc. cit.}. First, it is a continuous $\infty$-morphism of $\Linfty$-algebras by construction. Then, it is a filtered quasi-isomorphism by the same arguments as in the proof of 
\cite[Proposition~3.2]{DW14}: the chain complex ${\bar{\mathfrak{k}}}^{\, f}_{\calC, A, B}$ is isomorphic to the cylinder 
of the pushout map 
\[
\Hom_{\Sy}\left(\oC, \End_A\right) \to  \Hom_{\Sy}\left(\oC, \End^A_B\right) \ , \quad 
\alpha \mapsto f_*\circ \alpha
\]
and 
pullback map
\[
\Hom_{\Sy}\left(\oC, \End_B\right) \to  \Hom_{\Sy}\left(\oC, \End^A_B\right) \ , \quad 
\alpha \mapsto f^*\circ \beta~.
\]
Finally, one concludes with \cite[Lemma~1]{DW14}.
\end{proof}

Let us unravel this statement a little bit: given a quasi-isomorphism $f \colon A \xrightarrow{\sim} B$ of chain complexes and 
an $\Omega \calC$-gebra structure $\beta$ on $B$, there exists 
an $\Omega \calC$-gebra structure $\beta'$ on $B$ related to $\beta$ by two (strict) quasi-isomorphisms (by the arguments given in the proof of \cref{prop:InftyIsot}), 
a transferred $\Omega \calC$-gebra structure $\alpha$ on $A$, 
and an $\infty$-quasi-isomorphism $f+F$ from $\alpha$ to $\beta'$ extending $f$, that is
\[
\begin{tikzcd}
	(A, \alpha) \arrow[r, squiggly, "f+F"', "\sim"] & (B, \beta')
	& (\mathrm{Path}(B), \gamma) \ar[l,"\sim"', "g_B"] \ar[r,"\sim", "h_b"'] 
	& (B, \beta)~. 
\end{tikzcd}
\]
So, this version of the homotopy transfer theorem is very close to one given in \cite[Proposition~7.4]{Fresse10ter} whose upshot yields the existence of $\Omega \calC$-gebra structures and  quasi-isomorphisms 
\[
\begin{tikzcd}[column sep=small]
	(A, \alpha) \ar[r,"\sim", "i"'] 
	& (Z, \rho) 
	& (\mathrm{Path(Z)}, \varphi) \ar[r,"\sim", "h_Z"'] \ar[l,"\sim"', "g_Z"] 
	&(Z, \sigma) \ar[r,"\sim", "p"'] 
	& (B, \beta)\ ,
	\end{tikzcd} 
\]
where 
$\begin{tikzcd}
	A \ar[r,"\sim", "i"', hook] 
	&Z \ar[r,"\sim", "p"', two heads] 
	& B
	\end{tikzcd} $
is an acyclic fibration-cofibration factorisation of $f$~. 

\medskip

In the end, one can ``homotopically'' invert the quasi-isomorphisms going from right to left 
into $\infty$-quasi-isomorphisms using \cite[Theorem~4.18]{HLV19} and compose all the maps in order to get just one $\infty$-quasi-isomorphism from $(A, \alpha)$ to $(B, \beta)$:
\[
\begin{tikzcd}
	(A, \alpha) \arrow[r, squiggly, "\sim"] & (B, \beta)~.
\end{tikzcd}
\]
Notice that if one starts from a quasi-isomorphism $B \xrightarrow{\sim} A$ in the other way round, there always exists a quasi-isomorphism $A \xrightarrow{\sim} B$, since we are working over a field, and we can apply the abovementioned arguments.
These two versions of the homotopy transfer theorem are rather different to the one given in \cite[Theorem~4.14]{HLV19} by explicit formulas for the transferred structure $\alpha$ and the extension of $f$ into an $\infty$-quasi-isomorphism.

\bibliographystyle{alpha}
\bibliography{bib}

\end{document}

%% file: package.tex
\usepackage[english]{babel}

\usepackage[T1]{fontenc}
\usepackage[lf]{Baskervaldx} % lining figures
\usepackage[bigdelims,vvarbb]{newtxmath} % math italic letters from Nimbus Roman
\usepackage[cal=boondoxo]{mathalfa} % mathcal from STIX, unslanted a bit

\usepackage{a4wide}

\usepackage{latexsym}
\usepackage{tensor}
\usepackage{mathrsfs}
\usepackage{times,amsmath,amsthm}
\usepackage{graphicx}
\usepackage[all]{xy}
\usepackage{paralist}
\usepackage{subfigure}
\usepackage{multicol}
\usepackage{comment}

\usepackage{xcolor}
\definecolor{Chocolat}{rgb}{0.36, 0.2, 0.09}
\definecolor{BleuTresFonce}{rgb}{0.215, 0.215, 0.36}
\usepackage[colorlinks,final,hyperindex]{hyperref}
\hypersetup{
	pdftex,
	linkcolor=Chocolat,
	citecolor=BleuTresFonce,
	filecolor=BleuTresFonce,
	urlcolor=BleuTresFonce,
	pdftitle={Simplicial properadic homotopy},
	pdfauthor={E. Hoffbeck, J.Leray and B. Vallette},
	pdfsubject=,
	pdfkeywords=
}

\let\oldtocsection=\tocsection
\let\oldtocsubsection=\tocsubsection

\renewcommand{\tocsection}[2]{\hspace{0em}\vspace{0.1em}\rule{0pt}{14pt}\oldtocsection{#1}{#2}\bf}
\renewcommand{\tocsubsection}[2]{\hspace{2em}\oldtocsubsection{#1}{#2}}

\usepackage[noabbrev,capitalize]{cleveref}

\usepackage{tikz}
\usepackage{tikz-cd}
\usetikzlibrary{decorations.pathmorphing,arrows,arrows.meta,calc,shapes,shapes.geometric,decorations.pathreplacing, fit, positioning, matrix}
\tikzcdset{arrow style=tikz, diagrams={>=stealth}}

\tikzset{
  optree/.style={scale=.5,thick,grow'=up,level distance=10mm,inner sep=1pt},
  comp/.style={draw=none,circle,fill,line width=0,inner sep=0pt},
  dot/.style={draw,circle,fill,inner sep=0pt,minimum width=3pt},
  circ/.style={draw,circle,inner sep=1pt,minimum width=4mm},
  emptycirc/.style={draw,circle,inner sep=1pt,minimum width=2mm},
  root/.style={level distance=10mm,inner sep=1pt},
  leaf/.style={draw=none,circle,fill,line width=0,inner sep=0pt},
  nodot/.style={draw,circle,inner sep=1pt},
}

\usepackage{xparse}
\ExplSyntaxOn
\NewDocumentCommand{\createbunch}{ m O{} m }
 {
  \clist_map_inline:nn { #3 } { \cs_new_protected:cpn { #2 ##1 } { #1 { ##1 } } }
 }
\ExplSyntaxOff

\createbunch{\mathbb}   [bb]{A,B,C,D,E,F,G,H,I,J,K,L,M,N,O,P,Q,R,S,T,U,V,W,X,Y,Z}
\createbunch{\mathcal}  [cal]{A,B,C,D,E,F,G,H,I,J,K,L,M,N,O,P,Q,R,S,T,U,V,W,X,Y,Z}
\createbunch{\mathscr}  [scr]{A,B,C,D,E,F,G,H,I,J,K,L,M,N,O,P,Q,R,S,T,U,V,W,X,Y,Z}
\createbunch{\mathsf}   [sf]{A,B,C,D,E,F,G,H,I,J,K,L,M,N,O,P,Q,R,S,T,U,V,W,X,Y,Z}
\createbunch{\mathrm}   [rm]{A,B,C,D,E,F,G,H,I,J,K,L,M,N,O,P,Q,R,S,T,U,V,W,X,Y,Z}
\createbunch{\mathfrak} [frak]{A,B,C,D,E,F,G,H,I,J,K,L,M,N,O,P,Q,R,S,T,U,V,W,X,Y,Z}
\createbunch{\mathfrak} [frak]{a,b,c,d,e,f,g,h,i,j,k,l,m,n,o,p,q,r,s,t,u,v,w,x,y,z}
\theoremstyle{plain}
\newtheorem{thm}{Theorem}[section]
\newtheorem{proposition}[thm]{Proposition}
\newtheorem{theorem}[thm]{Theorem}
\newtheorem{corollary}[thm]{Corollary}

\newtheorem{lemma}[thm]{Lemma}
\newtheorem*{theointro}{Theorem}

\theoremstyle{definition}
\newtheorem{definition}[thm]{Definition}

\newtheorem{remark}[thm]{\sc Remark}
\newtheorem{exam}[thm]{\sc Example}

\newtheorem*{convention}{\sc Convention}

\def\gra{\mathrm{g}}
\def\mono{\hookrightarrow}
\def\epi{\twoheadrightarrow}
\def\cc{\circledcirc}
\def\scLi{\mathsf{scL}_\infty}
\def\point{\vcenter{\hbox{\scalebox{0.5}{$\bullet$}}}}
\def\k{\mathbb{k}}
\def\dgVect{\mathsf{dgVect}}

\def\M{\mathcal{M}}
\def\N{\mathcal{N}}
\newcommand{\Hom}{\ensuremath{\mathrm{Hom}}}
\renewcommand{\hom}{\ensuremath{\mathrm{hom}}}
\newcommand{\Cyl}{\ensuremath{\mathrm{Cyl}}}

 \newcommand{\eend}{\ensuremath{\mathrm{end}}}

\def\S{\mathrm{S}}

\def\I{\mathcal{I}}

\def\oC{\overline{\calC}}
\def\oD{\overline{\calD}}
\def\eps{\varepsilon}
\def\sq{\, \square\,}

\def\ibt{\underset{\scriptscriptstyle (1,1)}{\boxtimes}}
\def\libt{\lhd_{(*)}}
\def\ribt{\tensor*[_{(*)}]{\rhd}{}}

\def\DD{\ensuremath{\mathrm{D}}}
\def\UU{\ensuremath{\mathrm{U}}}
\def\h{\ensuremath{\mathfrak{h}}}

\def\H{\mathrm{H}}

\newcommand{\ac}{{\scriptstyle \text{\rm !`}}}

\def\L{\mathrm{L}}

\def\P{\mathcal{P}}

\def\id{\mathrm{id}}

\def\g{\mathfrak{g}}

\def\d{\mathrm{d}}

\def\Cop#1#2{\tensor*[_{#1}]{\Delta}{_{#2}}}

\newcommand{\ZZ}{\mathbb{Z}}
\newcommand{\NN}{\mathbb{N}}

\newcommand{\Map}{\mathrm{Map}}

\newcommand{\Sy}{{\mathbb{S}}}

\newcommand{\G}{\mathcal{G}}

\renewcommand{\Bar}{\mathbf{B}}

\newcommand{\Cobar}{\mathbf{\Omega}}

\newcommand{\Sbimod}{\Sy\mbox{-}\mathsf{bimod}}

\newcommand{\Set}{\mathsf{Sets}}

\renewcommand{\Sbimod}{\Sy\mbox{-}\mathsf{bimod}}

\renewcommand{\Set}{\mathsf{Set}}
\newcommand{\sSet}{\mathsf{sSet}}

\newcommand{\End}{\mathrm{End}}

\newcommand{\Linfty}{\mathrm{L}_\infty}

\newcommand{\cLi}{\mathsf{cL}_\infty}

\newcommand{\uCom}{\mathrm{uCom}}
\newcommand{\Com}{\mathrm{Com}}

\newcommand{\souche}{\mathcal{C}}

\renewcommand{\phi}{\varphi}

\newcommand{\too}{\longrightarrow}

\newcommand{\catofgebras}[1]{{#1}\text{-}\mathsf{gebras}}

\newcommand{\catofalgebras}[1]{{#1}\text{-}\,\mathsf{alg}}

\newcommand{\MC}{\mathrm{MC}}

%% file: authors.tex
\title{Simplicial properadic homotopy}

\date{\today}
\author{Eric Hoffbeck}
\address{LAGA, CNRS, UMR 7539, Universit\'e Sorbonne Paris Nord, Universit\'e Paris 8, 99 Avenue Jean-Baptiste Cl\'ement, 93430 Villetaneuse, France.}
\email{hoffbeck@math.univ-paris13.fr}
\email{vallette@math.univ-paris13.fr}

\author{Johan Leray}
\address{Universit\'e de Nantes, 
Laboratoire de Math\'ematiques Jean LERAY (LMJL), CNRS, 
UMR 6629, 
2 Chemin de la Houssini\`ere, BP 92208,  44322 Nantes Cedex 3, France}
\email{johan.leray@univ-nantes.fr}

\author{Bruno Vallette}

\subjclass[2010]{Primary 18M85; 
Secondary 14D15, 16T10, 17B55, 18M70, 18N40, 18N50}

\keywords{Homotopical algebra, bialgebras, properads, $\infty$-morphisms, simplicial category}

\thanks{The authors were supported by the ANR HighAGT (ANR-20-CE40-0016) grant.}